\documentclass{article}

\usepackage{microtype}
\usepackage{verbatim}
\usepackage{graphicx}
\usepackage{rotating}
\usepackage{listings}

\usepackage{url}
\usepackage{hyperref}

\usepackage{pslatex}

\usepackage{xpatch} 

\usepackage{amsthm}
\usepackage[fleqn]{amsmath} 
\usepackage{amsfonts}
\usepackage{amssymb}

\newtheoremstyle{zoltanstyle}
  {1em} 
  {\topsep} 
  {} 
  {} 
  {\bfseries} 
  {.} 
  {.5em} 
  {} 

\theoremstyle{zoltanstyle}
\swapnumbers
\xpatchcmd\swappedhead{~}{.~}{}{}
\newtheorem{body}{}
\numberwithin{body}{section}

\newtheorem{corollary}[body]{Corollary}
\newtheorem{definition}[body]{Definition}
\newtheorem{example}[body]{Example}

\newtheorem{lemma}[body]{Lemma}

\newtheorem{proposition}[body]{Proposition}

\newtheorem{theorem}[body]{Theorem}

\expandafter\let\expandafter\oldproof\csname\string\proof\endcsname
\let\oldendproof\endproof

\renewenvironment{proof}[1][\proofname]{%
  \oldproof[\normalfont \bfseries #1.]%
}{\oldendproof}

\usepackage{enumitem}
\usepackage{mathptmx}
\usepackage{xspace}

\usepackage[T1]{fontenc}
\usepackage[utf8]{inputenc}

\newcommand{\SetComp}[2]{\left\{ {#1}\:\middle|\:{#2} \right\}}   

\title{Distributive lattices in o-minimal structures}
\author{Zoltan A. Kocsis\footnote{University of New South Wales, Kensington NSW 2052, Australia.}}
\date{26 June 2026} 

\begin{document}

\maketitle

\begin{abstract}
We investigate distributive lattices and Heyting algebras definable in o-minimal structures. We give a complete description of one-dimensional distributive lattices definable in o-minimal structures expanding a real-closed field, and prove a definable analogue of Birkhoff representation. As applications, we obtain a sharp lower bound on the dimension of definable Heyting algebras which contain infinite free subalgebras, and determine all one-variable equations in the language of Heyting algebras whose solution set can constitute a maximal-dimension proper subset of an infinite algebra.
\end{abstract}

This article contributes to the long-standing program of characterizing the infinite algebraic structures definable in o-minimal structures, which started with Pillay's~\cite{pillay-groups} results on groups and fields definable in o-minimal structures. Later, Otero,~Peterzil~and~Pillay \cite{otero-peterzil-rings} showed that an infinite ring without zero divisors definable in an o-minimal expansion of a real-closed field $R$ is definably isomorphic to one of $R$, $R[i]$, or the ring of quaternions over $R$. The literature is extensive: for more recent results on groups and rings, see e.g. Otero's survey~\cite{otero-survey}. Here we study the definability question for distributive lattices and Heyting algebras.

The motivation for the present work is the desire to understand whether tame distributive lattices can be extended with further ``tame'' operations. For example, every distributive lattice embeds into a Boolean algebra, but this embedding is not tame: o-minimal structures contain definable infinite distributive lattices, but not definable infinite Boolean algebras. The author wondered what (if any) pathologies could prevent a distributive lattice, definable in an o-minimal structure, from embedding into at least a Heyting algebra definable in the same structure.

Our main contributions are as follows. We obtain a complete combinatorial classification of one-dimensional bounded distributive lattices definable in an o-minimal expansion of a real-closed field (Theorem~\ref{thm:tame-variation}). This lets us prove a definable, dimension one analogue (Theorem~\ref{thm:birkhoff-embedding}) of the classic result that every distributive lattice embeds into some Heyting algebra, whose more usual proof strategies, such as Stone representation, are blocked in the o-minimal setting.

\vspace{0.5em} \textbf{Outline.} In Section~\ref{sec:preliminaries}, we introduce common notation and conventions, then collect known order-theoretic results required by our main development. These include Ramakrishnan's characterization of definable linear orders in o-minimal structures, and the fact that definable Boolean algebras are finite. We give a self-contained proof of the latter statement using only elementary o-minimal dimension theory rather than NIP/VC machinery.

\vspace{0.5em} In Section~\ref{sec:distributive-lattices-of-dim-one}, we study the structure of one-dimensional bounded distributive lattices. We show the definability of Artin gluing, and use it to obtain a large class of one-dimensional definable Heyting algebras. We then classify (up to definable isomorphism) the one-dimensional definable bounded distributive lattices (Theorem~\ref{thm:tame-variation}).

\vspace{0.5em} In Section~\ref{sec:birkhoff-style-results}, we show that each poset which admits a definable chain decomposition embeds definably into a Heyting algebra (Theorem~\ref{thm:dcr-for-chain-decomposed-orders}). Making heavy technical use of the structure theory developed in the previous section, we construct join-irreducibly generated completions of one-dimensional distributive lattices~(Theorem~\ref{thm:jig-embedding}). This leads immediately to a definable analogue (Theorem~\ref{thm:birkhoff-embedding}) of the classic result that every distributive lattice embeds into a Heyting algebra.

\vspace{0.5em} Finally, in Section~\ref{sec:applications}, we give two applications of the previous results which shed some light on the algebraic properties of infinite definable distributive lattices. We show the existence of definable Heyting algebras which contain free subalgebras~(Example~\ref{example:definable-ha-with-free-subalgebra}), and determine which one-variable systems of equations can have solution sets of dimension $n$ in an $n$-dimensional Heyting algebra.

\section{Preliminaries}\label{sec:preliminaries}

\begin{body}
For the remainder of this article we fix an arbitrary o-minimal structure $$(R, \leq, +, -, \times, 0, 1)$$ expanding a real-closed field. In this section we define our notation and recall the results that we will invoke in our main development. As customary in model theory, we do not distinguish between nested tuples of elements of $R$, i.e. we freely identify $(R^n)^m$ with $R^{nm}$ throughout the text. Unless indicated otherwise, we use the conventions of the textbook ``Tame Topology and O-minimal Structures''~\cite{vandendries-tame}. In particular, we allow parameters in formulae, and take \textit{definable set} to mean a set \textit{definable with parameters}.
\end{body}

\begin{definition}\label{def:definable}
Consider a first-order language $\mathcal{L}$ and an $\mathcal{L}$-structure $(M,\dots)$. Recall that if
\begin{itemize}
    \item $M$ is a definable subset of $R^n$,
    \item for each predicate symbol $P \in \mathcal{L}$ of arity $a$, the set $$\SetComp{(x_1,\dots,x_a)\in M^a}{P(x_1,\dots,x_a)}$$ is a definable subset of $R^{an}$, and
    \item for each function symbol $f \in \mathcal{L}$ of arity $a$, the set $$\SetComp{(x_1,\dots,x_a,y)\in M^{a+1}}{f(x_1,\dots,x_a)=y}$$ is a definable subset of $R^{(a+1)n}$,
\end{itemize}
then we say that \textit{$M$ is a definable $\mathcal{L}$-structure in the o-minimal structure $R$}.
\end{definition}

\begin{body}
From here onward, \textit{definable structure} refers to a structure definable in this particular $R$ in the sense of Definition~\ref{def:definable}.
\end{body}

\subsection*{Definable partial orders}

\begin{definition}\label{def:definable-poset}
Consider a definable partial order $(L, \sqsubseteq)$ and elements $x,y \in L$ with $x \sqsubset y$. We call the set
$$\SetComp{z \in L}{x \sqsubset z \wedge z \sqsubset y}$$ 
the \textit{$\sqsubseteq$-interval} given by $x,y$ and denote it $(x,y)_\sqsubseteq$. We define the notations $\left(x,y\right]_\sqsubseteq$, $\left[x,y\right)_\sqsubseteq$ and $\left[x,y\right]_\sqsubseteq$ for the other $\sqsubseteq$-convex subsets of $L$ the obvious way.
\end{definition}

\begin{definition}
We call a definable partial order $(L, \sqsubseteq)$ \textit{homogeneous $n$-dimensional} if $\dim L = n$ and every $\sqsubseteq$-interval has dimension exactly $n$.
\end{definition}

\begin{body}
In an o-minimal structure, one can write every definable subset of $R$ as a finite union of intervals and singleton points. In Corollary~\ref{cor:homogeneous-decomposition} we show the analogous property for any homogeneous one-dimensional linear order $(L, \sqsubseteq)$: the subsets of $L$ decompose into $\sqsubseteq$-intervals and points. Our proof uses a classic result of Ramakrishnan~\cite{ramakrishnan-linear}, restated for completeness as Theorem~\ref{thm:ramakrishnan-linear}. One cannot do away with the homogeneity condition, as shown in Example~\ref{example:homogeneous-decomposition-required}.
\end{body}

\begin{theorem}[Ramakrishnan~\cite{ramakrishnan-linear}]\label{thm:ramakrishnan-linear}
Every definable linear order $(L, \sqsubseteq)$ order-embeds definably into $(R^{1 + \dim L},\leq_{lex})$ where $\leq_{lex}$ denotes the lexicographic order.
\end{theorem} 

\begin{proposition}\label{prop:ordered-sums-preserve-r}
Given two definable linear orders $L_1$ and $L_2$ that are definable order-isomorphic to a definable subset of $(R,\leq)$, the ordered sum of $L_1$ and $L_2$ also admits a definable order-isomorphism to some subset of $(R, \leq)$.
\end{proposition}

\begin{proposition}\label{prop:homogeneous-order-embedding}
Every homogeneous one-dimensional definable linear order $(L, \sqsubseteq)$ is order-isomorphic to a definable subset of $(R, \leq)$.
\begin{proof}
By Ramakrishnan's result (Theorem~\ref{thm:ramakrishnan-linear} above), the definable linear order $(L, \sqsubseteq)$ embeds definably into $(R^{1 + \dim L},\leq_{lex})$ where $\leq_{lex}$ denotes the lexicographic order. Thus, we may assume without loss of generality that $L \subseteq R^2$ and $\sqsubseteq$ denotes the lexicographic order restricted to $L$.
Consider the definable sets of points
$$ E_\infty = \SetComp{x \in R}{\exists y_1,y_2\in R.y_1\neq y_2 \wedge (x,y_1) \in L \wedge (x, y_2) \in L},$$
$$ E_{1} = \SetComp{x \in R}{\exists! y \in R. (x, y) \in L}.$$
Notice that the first coordinate of any element of $L$ belongs to one of these sets and if $x \in E_\infty$, then by the homogeneity condition infinitely many $y$ satisfy $(x,y) \in L$. Thus, $E_\infty$ consists of only finitely many points: otherwise it would contain some interval, and so one would have $\dim \SetComp{(x,y) \in L}{x \in E_\infty} \geq 2$, contradicting that $\dim L = 1$.
Decompose $E_1 \cup E_\infty$ into a finite sequence of singleton points and intervals. By the lexicographic property, $(L, \sqsubseteq)$ arises as the ordered sum of these finitely many singletons and intervals. Each interval belongs to $E_1$, and when equipped with the induced order, admits a definable order-isomorphism to some subset of $(R, \leq)$. For each fixed $x \in E_\infty$, the fiber $L_x = \SetComp{y \in R}{(x,y) \in L}$ also embeds into $(0,1) \subseteq R$ with the usual ordering.  Applying Proposition~\ref{prop:ordered-sums-preserve-r} lets us conclude that $(L, \sqsubseteq)$ admits a definable order-isomorphism to a subset of $(R, \leq)$.
\end{proof}
\end{proposition}

\begin{corollary}\label{cor:homogeneous-decomposition}
Consider a homogeneous one-dimensional linear order $(L, \sqsubseteq)$, along with a definable function $f: L \rightarrow D$ to some finite definable set $D$. Then one can decompose $L$ into finitely many singleton points and $\sqsubseteq$-intervals on each of which $f$ is constant.
\end{corollary}

\begin{example}\label{example:homogeneous-decomposition-required}
Consider the set $L =\SetComp{(x,y) \in R^2}{x = 1 \vee x = 2}$ ordered by the relation $(x,y) \sqsubseteq (x',y')$ which holds when $y' > y$ or when $y' = y$ and $x' \geq x$. Then one cannot decompose $L$ into finitely many singleton points and $\sqsubseteq$-intervals on which the projection function $\pi(x,y) = x$ is constant. 
\end{example}

\begin{body}
It also follows from Theorem~\ref{thm:ramakrishnan-linear} that all definable linear orders $(L, \sqsubseteq_L)$ admit a \textit{definable completion}, in other words a definable embedding into some linear order $(\overline{L}, \sqsubseteq)$ in which every definable set of elements of $L$ has a unique least upper bound in $\overline{L}$. Throughout the article, we permit the interval notations of Definition~\ref{def:definable-poset} to take their endpoints in the definable completion.
\end{body}

\subsection*{Definable lattices}

\begin{definition}\label{def:lattice}
We call a partial order $(H, \sqsubseteq)$ a \textit{lattice} if every two-element subset $\{x,y\} \subseteq H$ has a least upper bound $x \sqcup y$ and a greatest lower bound $x \sqcap y$. If furthermore the lattice $(H, \sqsubseteq)$ has an $\sqsubseteq$-maximum element $\top$ and $\sqsubseteq$-minimum element $\bot$, we call it a \textit{bounded lattice}.
\end{definition}

\begin{proposition}\label{prop:lattice-equidefinable}
In a bounded lattice $(H, \sqsubseteq)$, if any one of the following is definable, then so are the other two,
\begin{enumerate}
    \item the least upper bound map $(x,y) \mapsto x \sqcup y$ as a function $H^2 \rightarrow H$,
    \item the greatest lower bound map $(x,y) \mapsto x \sqcap y$ as a function $H^2 \rightarrow H$,
    \item the order relation $\sqsubseteq$ as a definable subset of $H^2$.
\end{enumerate}
\begin{proof}
Least upper bounds admit a first-order definition in terms of $\sqsubseteq$. In a lattice $x \sqsubseteq y$ holds precisely if $x \sqcap y = x$, equivalently if $x \sqcup y = y$.
\end{proof}
\end{proposition}

\begin{body}\label{body:lattice-notation}
In light of Proposition~\ref{prop:lattice-equidefinable}, we adopt the following notational conventions. In any result concerning a definable (bounded, distributive, etc.) lattice $(X, \sqsubseteq_X)$, we let $\sqcup_X$ stand for the least upper bound map definable from $\sqsubseteq_X$. Similarly, $\sqcap_X$ stands for the corresponding greatest lower bound map, $\top_X$ refers to the $\sqsubseteq$-maximum element of $X$, and $\bot_X$ to the $\sqsubseteq$-minimum element of $X$. We also extend the $\sqsubseteq$-interval notation of Definition~\ref{def:definable-poset}, so that $(x,y)_X$ denotes the $\sqsubseteq_X$-interval between the elements $x$ and $y$. For example, in a result concerning the lattices $(X, \sqsubseteq_X)$, $(Y, \sqsubseteq_Y)$ and $(H, \sqsubseteq)$,
\begin{itemize}
    \item the symbol $\sqcap_X$ denotes the greatest lower bound function $X^2 \rightarrow X$ with respect to the partial order $\sqsubseteq_X$,
    \item the symbol $\bot_Y$ denotes the minimum element of $Y$ with respect to the partial order $\sqsubseteq_Y$,
    \item the symbol $\sqcup$ denotes the least upper bound function $H^2 \rightarrow H$ with respect to the partial order $\sqsubseteq$,
    \item the formula $[a,b]_X$ stands for the set $\SetComp{x \in X}{a \sqsubseteq_X x \wedge x \sqsubseteq_X b}$
\end{itemize}
and so on. Keep in mind however that we use the symbols $\wedge$ / $\vee$ in their logical sense (conjunction / disjunction) only, never as least upper bound operations of the definable order $\leq$.
\end{body}

\begin{definition}\label{def:heyting-algebra}
We call a lattice $(H, \sqsubseteq)$ \textit{distributive} if the algebraic identity $$x \sqcap (y \sqcup z) = (x \sqcap y) \sqcup (x \sqcap z)$$ holds for all $x,y,z \in H$. In turn, we call the bounded distributive lattice $(H, \sqsubseteq)$ a \textit{Heyting algebra} if the set $\SetComp{z \in H}{x \sqcap z \sqsubseteq y}$ has a $\sqsubseteq$-maximum element for each $x,y \in H$. If so, we denote this unique element by $x \Rightarrow y$, and call $\Rightarrow$ the \textit{implication operation} of the Heyting algebra $(H, \sqsubseteq)$.
\end{definition}

\begin{proposition}\label{prop:heyting-equidefinable}
In a Heyting algebra $(H, \sqsubseteq)$, if any one of the following are definable, then so are the other three:
\begin{enumerate}
    \item the implication operation $(x,y) \mapsto x \Rightarrow y$ as a function $H^2 \rightarrow H$,
    \item the least upper bound map $(x,y) \mapsto x \sqcup y$ as a function $H^2 \rightarrow H$,
    \item the greatest lower bound map $(x,y) \mapsto x \sqcap y$ as a function $H^2 \rightarrow H$,
    \item the order relation $\sqsubseteq$ as a definable subset of $H^2$.
\end{enumerate}
\begin{proof}
The equivalence of 2-4 follows from Proposition~\ref{prop:lattice-equidefinable}. The operation $\Rightarrow$ was defined in terms of the order in Definition~\ref{def:heyting-algebra}. Conversely, $\Rightarrow$ lets us define $\sqsubseteq$ using the fact that $(x \Rightarrow y) = (x \Rightarrow x)$ holds precisely if $x \sqsubseteq y$.
\end{proof}
\end{proposition}

\begin{body}
We assume that the reader is familiar with distributive lattice concepts such as join-irreducibility: Chapter~III of the textbook ``Practical Foundations of Mathematics''~\cite{taylor-pfom} covers all order theory prerequisites. For a comprehensive, contemporary reference on Heyting algebras, we refer the reader to Esakia's textbook~\cite{esakia-heyting}. Since verifying that a partial order $(H, \sqsubseteq)$ satisfies the definition of Heyting algebra above can get tedious, we include a well-known algebraic characterization of Heyting algebras in Proposition~\ref{prop:equational-heyting-algebra} below.
\end{body}

\begin{proposition}\label{prop:equational-heyting-algebra}
The structure $(H, \sqcap, \sqcup, \bot, \top)$ constitutes a bounded lattice precisely if $\sqcap, \sqcup$ are commutative and associative binary operations, and the following identities hold for all $x,y,z \in H$:
\begin{enumerate}
    \item $\top \sqcap x = x = \bot \sqcup x$,
    \item $\top \sqcup x = \top$,
    \item $\bot \sqcap x = \bot$, and
    \item $x \sqcap (y \sqcup x) = x = x \sqcup (y \sqcap x)$.
\end{enumerate}
Given a binary operation $\Rightarrow : H^2 \rightarrow H$, the structure $(H, \sqcap, \sqcup, \Rightarrow, \bot, \top)$ constitutes a Heyting algebra precisely if the following four identities, as well as all the bounded lattice identities above, hold for all $x,y,z \in H$:
\begin{enumerate}
    \item $x \Rightarrow x = \top$,
    \item $x \sqcap (x \Rightarrow y) = x \sqcap y$,
    \item $y \sqcap (x \Rightarrow y) = y$, and
    \item $x \Rightarrow (y \sqcap z) = (x \Rightarrow y) \sqcap (x \Rightarrow z)$.
\end{enumerate} 
\end{proposition}

\begin{example}\label{example:boolean-as-heyting}
Every Boolean algebra $(H, \sqcap, \sqcup, \bot, \top, \neg_H)$ is a Heyting algebra with implication operation $x \Rightarrow y$ given as $\neg_H x \sqcup y$. Conversely, a Heyting algebra $H$ is a Boolean algebra precisely if $x \Rightarrow y = (x \Rightarrow \bot) \sqcup y$ for all $x,y \in H$.
\end{example}

\begin{body}
We extend the notational conventions of Section~\ref{body:lattice-notation} to Heyting algebras: given a Heyting algebra $(X,\sqsubseteq_X)$, we let $\Rightarrow_X$ denote the corresponding implication operation. As customary in the Heyting algebra literature, we also introduce the abbreviated form $\neg_X a$ denoting the term $a \Rightarrow_X \bot$. Similarly to $\wedge$ / $\vee$, we retain the thin arrow notation $\rightarrow$ as \textit{logical} implication, and do not use it to denote the implication operation of a definable Heyting algebra.
\end{body}

\begin{proposition}\label{prop:adjoint-functor-thm}
Consider two posets $(P, \sqsubseteq_P)$ and $(Q, \sqsubseteq_Q)$ in which every subset has a least upper bound, and a monotone map $l: P \rightarrow Q$ that preserves all joins. Then $l$ admits a right adjoint $r: Q \rightarrow P$ given by the formula
$$r(q) = \textstyle \bigsqcup_P \SetComp{p \in P}{l(p)\sqsubseteq_Q q}.$$
\begin{proof}
See Theorem~3.6.9~of~\cite{taylor-pfom}. Consequence of the adjoint functor theorem instantiated for posets.
\end{proof}
\end{proposition}

\begin{body}
Every finite distributive lattice is definable. Propositions~\ref{prop:join-irreducible-antichains},~\ref{prop:finite-cover-property}~and~\ref{prop:boolean-algebras-are-finite} give basic constraints on definability of infinite distributive lattices. All three are well-known and admit straightforward proofs using more advanced model-theoretic machinery (NIP or the Vapnik-Chervonenkis property). However, to the author's best knowledge, the arguments presented below, which use only elementary dimension theory presented in Chapters 1-4 of~\cite{vandendries-tame}, are novel.
\end{body}

\begin{lemma}\label{lemma:irreducible-forks}
Consider a distributive lattice $(H, \sqsubseteq)$. Let $x,y,j \in H$ so that $j \sqsubseteq x \sqcup y$ and $j$ is join-irreducible. Then either $j \sqsubseteq x$ or $j \sqsubseteq y$ holds.
\begin{proof}
Write $j = j \sqcap (x \sqcup y) = (j \sqcap x) \sqcup (j \sqcap y)$, then use join-irreducibility to conclude that one of $j = j \sqcap x$ or $j = j \sqcap y$ must hold.
\end{proof}
\end{lemma}

\begin{proposition}\label{prop:join-irreducible-antichains}
No definable distributive lattice admits an infinite definable antichain of join-irreducible elements.
\begin{proof}
Assume for a contradiction that the definable distributive lattice $(H, \sqsubseteq)$ contains an infinite antichain $A \subseteq H$ of join-irreducibles. Take a cell-decomposition of $H$ partitioning $A$, and extract a one-dimensional subcell $a: (0,1) \hookrightarrow A$. Set
$$I_k = \SetComp{(x_1,\dots,x_k)\in (0,1)^k}{x_1 < x_2<\cdots<x_k}.$$
Define the map $\iota_k : I_k \rightarrow H$ using the formula $$\iota_k(x_1,\dots,x_k) = a(x_1) \sqcup a(x_2) \sqcup \dots \sqcup a(x_k).$$
We claim that $\iota_k$ is injective. Assume that $\iota_k(x_1,\dots,x_k) = \iota_k(y_1,\dots,y_k)$. Then, by Lemma~\ref{lemma:irreducible-forks}, $a(t) \sqsubseteq \iota_k(x_1,\dots,x_k)$ precisely if $a(t) \sqsubseteq a(x_i)$ for some index $i$. Similarly, $a(t) \sqsubseteq a(y_j)$ for some index $j$. Since the image of $a$ is an antichain, $a(t) \sqsubseteq a(x_i)$ only if $a(t) = a(x_i)$, which implies $t = x_i$ using the injectivity of $a$. This means that the sets $\{x_1,\dots,x_k\}$ and $\{y_1,\dots,y_k\}$ coincide, and by definition of $I_k$ so do the tuples $(x_1,\dots,x_k)$ and $(y_1,\dots,y_k)$, showing injectivity of $\iota_k$. From injectivity it follows that $k = \dim I_k \leq \dim H$. Taking $k > \dim H$ leads to a contradiction. 
\end{proof}
\end{proposition}

\begin{proposition}\label{prop:finite-cover-property}
The set $\SetComp{y \in H}{\#[x,y]_\sqsubseteq=2}$ of elements covering $x \in H$ (resp. the set $\SetComp{y \in H}{\#[y,x]_\sqsubseteq=2}$ of elements covered by $x \in H$) is  finite in every definable distributive lattice $(H,\sqsubseteq)$.
\begin{proof}
Consider the distributive lattice obtained by restricting $\sqsubseteq$ to $\SetComp{y \in H}{x \sqsubseteq y}$. In the resulting lattice, the points of $H$ covering $x$ are atoms, and hence join-irreducible. Since atoms are incomparable, they form an antichain. Consequently, by Proposition~\ref{prop:join-irreducible-antichains}, there are only finitely many of them.
\end{proof}
\end{proposition}

\begin{proposition}\label{prop:diamond-split}
Given a definable distributive lattice $(H,\sqsubseteq)$ and a pair of elements $a,b \in H$, we have a (definable) bijection between the $\sqsubseteq$-interval $[a \sqcap b, a \sqcup b]_\sqsubseteq$ and the Cartesian product $[a \sqcap b, a]_\sqsubseteq \times [a \sqcap b, b]_\sqsubseteq$.
\begin{proof}
The definable map $f(x) = (x \sqcap a, x \sqcap b)$ admits the inverse $f^{-1}(x,y) = x \sqcup y$.
\end{proof}
\end{proposition}

\begin{corollary}\label{cor:central-split}
Consider a definable Heyting algebra $(H, \sqsubseteq)$ with an element $h \in H$ satisfying $h \sqcup \neg h = \top$. Then the intervals $[\bot, h]_\sqsubseteq$ and $[\bot, \neg h]_\sqsubseteq$ form Heyting algebras under the order inherited from $H$, and we have a definable isomorphism between $H$ and the direct product of Heyting algebras $[\bot, h]_\sqsubseteq \times [\bot, \neg h]_\sqsubseteq$.
\begin{proof}
It is immediate that $D=[\bot, h]_\sqsubseteq$ and $E=[\bot, \neg h]_\sqsubseteq$ form bounded distributive lattices under the inherited order. Moreover, $x \sqcap y = x \sqcap_D y$ when $x,y \in D$, and similarly for $\sqcap_E$ on $E$.

\vspace{0.5em} \textbf{Claim 1.} Setting $x \Rightarrow_D y = h \sqcap (x \Rightarrow y)$ turns $D$ into a Heyting algebra (and similarly for $E$).
Take $x,y,z \in D$, then argue as follows:
\begin{align*}
    x \sqsubseteq_D y \Rightarrow_D z &\leftrightarrow x \sqsubseteq h \sqcap (y \Rightarrow z) & \\
    &\leftrightarrow x \sqsubseteq y \Rightarrow z & \text{(since $x \in [\bot,h]_\sqsubseteq$)}\\
    &\leftrightarrow x \sqcap y \sqsubseteq z & \\
    &\leftrightarrow x \sqcap_D y \sqsubseteq_D z & \text{(since $\sqcap_D=\sqcap$)}.
\end{align*}

\vspace{0.5em} \textbf{Claim 2.} The definable maps $f$ and $f^{-1}$ of Proposition~\ref{prop:diamond-split} give the required isomorphism of Heyting algebras. Preservation of the lattice operations follows from a routine verification. The fact that $f$ preserves the Heyting structure is immediate from Claim 1. We only need to check that the inverse $f^{-1}$ preserves the Heyting structure as well. But we have
\begin{align*}
    f^{-1}(x \Rightarrow_D y, x \Rightarrow_E y) &= (x \Rightarrow_D y) \sqcup (x \Rightarrow_E y) & \\
    &= (x \Rightarrow_D y) \sqcup (\neg h \sqcap (x \Rightarrow y)) & \\
    &= ((x \Rightarrow_D y) \sqcup \neg h) \sqcap ((x \Rightarrow_D y) \sqcup (x \Rightarrow y)) & \text{(distributivity)}\\
    &= ((x \Rightarrow_D y) \sqcup \neg h) \sqcap (x \Rightarrow y) & \text{(absorption)}\\
    &= ((h \sqcup \neg h) \sqcap (\neg h \sqcup (x \Rightarrow y))) \sqcap (x \Rightarrow y) & \text{(distributivity)}\\
    &= (\neg h \sqcup (x \Rightarrow y)) \sqcap (x \Rightarrow y) & \text{($h$ central)}\\
    &= x \Rightarrow y & \text{(absorption)}\\
\end{align*}
and the result follows.
\end{proof}
\end{corollary}

\begin{proposition}\label{prop:boolean-algebras-are-finite}
Every definable Boolean algebra $(B, \sqsubseteq)$ is finite.
\begin{proof}
Take an infinite definable Boolean algebra $(B, \sqsubseteq)$  of smallest possible dimension. Let $d(x) = \dim [\bot,x]_\sqsubseteq$. Then, by Chapter~4,~Proposition~1.5~of~\cite{vandendries-tame}, $x \mapsto d(x)$ is a definable function.

\vspace{0.5em}\textbf{Claim.} For all $x \in B$, $d(x) = 0$ or $d(\neg x) = 0$. Otherwise we would have some $x \in B$ so that $d(x) > 0$ and $d(\neg x) > 0$. By Corollary~\ref{cor:central-split}, $\dim B = d(x) + d(\neg x)$, so this would split $B$ into the product of two strictly smaller-dimensional infinite Boolean algebras. We chose $(B, \sqsubseteq)$ to have the smallest possible dimension for an infinite Boolean algebra, which makes this impossible.

\vspace{0.5em} Now consider the definable relation
$$R(x,y) \leftrightarrow (d(x) = 0 \wedge y \sqsubseteq x) \vee (d(\neg x) = 0 \wedge y \sqsubseteq \neg x).$$ By Chapter~3, Corollary~3.7~of~\cite{vandendries-tame}, we can find $N \in \mathbb{N}$ so that $\#\SetComp{y \in B}{R(x,y)} \leq N$ for any $x \in B$. Thus, $B$ has finite height, and is therefore finite, contradicting our initial assumption.
\end{proof}
\end{proposition}

\section{Distributive lattices of dimension one}\label{sec:distributive-lattices-of-dim-one}

\begin{body}
In this section, after a brief survey of common examples and constructions, we give a complete description of the structure of one-dimensional bounded distributive lattices definable in o-minimal fields (Theorem~\ref{thm:tame-variation}). In later sections, we apply the structure theory developed here to show that definable one-dimensional distributive lattices embed definably into definable Heyting algebras (Theorem~\ref{thm:birkhoff-embedding}).
\end{body}

\subsection*{Artin gluing}

\begin{definition}\label{def:artin-gluing}
Consider a Heyting algebra $(H,\sqcup_H,\sqcap_H,\Rightarrow_H,\bot_H,\top_H)$, a finite Heyting algebra $(F,\sqcup_F,\sqcap_F,\Rightarrow_F,\bot_F,\top_F)$ and a function $\tau : H \rightarrow F$ satisfying $\tau(\top_H) = \top_F$ and $\tau(a \sqcap_H b) = \tau(a) \sqcap_F \tau(b)$ for any $a,b \in H$. Define
\begin{itemize}
    \item the set $\SetComp{(h,f)\in H\times F}{f \sqsubseteq_F \tau(h)}$, which we denote as $H \ltimes_\tau F$,
    \item the binary operations $\sqcap$ and $\sqcup$ on $H \ltimes_\tau F$ pointwise,
    \item the operation $\Rightarrow$ given by the equation $$(h_1,f_1) \Rightarrow (h_2, f_2) = (h_1 \Rightarrow_H h_2, (f_1 \Rightarrow_F f_2) \sqcap_F \tau(h_1 \Rightarrow_H h_2)).$$
\end{itemize}
We call the resulting structure on $H \ltimes_\tau F$ the \textit{Artin gluing} of the algebras $H$ and $F$ along the map $\tau$.
\end{definition}

\begin{proposition}\label{prop:artin-lattice}
The binary operations of Definition~\ref{def:artin-gluing} have signature $(H \ltimes_\tau F)^2 \rightarrow H\ltimes_\tau F$ and make the set $H \ltimes_\tau F$ into a bounded distributive lattice.
\begin{proof}
We only prove well-definedness of $\sqcup$, since the existence of top and bottom elements, as well as the other well-definedness conditions and bounded distributive lattice properties follow by straightforward algebraic reasoning. Notice the monotonicity of $\tau$: $x \sqsubseteq_H y$ holds precisely if $x \sqcap_H y = x$, and then $\tau(x) \sqcap_F \tau(y) = \tau(x)$ and so $\tau(x) \sqsubseteq_F \tau(y)$. Assume that $H\ltimes_\tau F$ contains $(h_1,f_1)$ and $(h_2, f_2)$, in other words that $f_1 \sqsubseteq_F \tau(h_1)$ and $f_2 \sqsubseteq_F \tau(h_2)$. To show the well-definedness of $\sqcup$, we need to prove that $f_1 \sqcup_F f_2 \sqsubseteq_F \tau(h_1 \sqcup_H h_2)$. By monotonicity of $\tau$ and $h_1 \sqsubseteq_H h_1 \sqcup_H h_2$ we have $f_1 \sqsubseteq_F \tau(h_1) \sqsubseteq_F \tau(h_1 \sqcup_H h_2)$ and similarly $f_2 \sqsubseteq_F \tau(h_2) \sqsubseteq_F \tau(h_1 \sqcup_H h_2)$. Thus, by the property of least upper bounds we get $f_1 \sqcup_F f_2 \sqsubseteq_F \tau(h_1 \sqcup_H h_2)$.
\end{proof}
\end{proposition}

\begin{proposition}\label{prop:artin-algebra}
The Artin gluing of two (definable) Heyting algebras along a (definable) map $\tau$ yields a (definable) Heyting algebra.
\begin{proof}
We use the notations of Definition~\ref{def:artin-gluing}. Definability follows immediately from there. Since we treat the distributive lattice properties in Proposition~\ref{prop:artin-lattice}, we only have to check that the four axioms involving $\Rightarrow$ hold. Given $(h_i,f_i)$ with $f_i \sqsubseteq_F \tau(h_i)$, we verify the properties in order.
\begin{enumerate}
    \item $(h_1, f_1) \Rightarrow (h_1, f_1) = \top$: we need to check that $(f_1 \Rightarrow_F f_1) \sqcap_F \tau(h_1 \Rightarrow_H h_1) = \top_F$, which follows since $\top_F \sqcap_F \tau(h_1 \Rightarrow_H h_1) = \tau(h_1 \Rightarrow_H h_1) = \tau(\top_H) = \top_F$.
    \item $(h_1, f_1) \sqcap ((h_1,f_1) \Rightarrow (h_2, f_2)) = (h_1,f_1) \sqcap (h_2,f_2)$: again, we only need to check that the second coordinates agree. This means $f_1\sqcap_F (f_1 \Rightarrow_F f_2) \sqcap \tau(h_1 \Rightarrow_H h_2) = f_1 \sqcap f_2$. But we have $f_2 \sqsubseteq_F \tau(h_2) \sqsubseteq \tau(h_1 \Rightarrow_H h_2)$ by the monotonicity of $\tau$ and $h_2 \sqsubseteq_H h_1 \Rightarrow h_2$. Thus, $f_2 \sqcap \tau(h_1 \Rightarrow_H h_2) = f_2$ and the result follows.
    \item $(h_2, f_2) \sqcap ((h_1,f_1) \Rightarrow (h_2, f_2)) = (h_2,f_2)$: follows from the monotonicity of $\tau$, similarly to the previous identity.
    \item $(h_1,f_1) \Rightarrow ((h_2,f_2) \sqcap (h_3,f_3)) = ((h_1,f_1) \Rightarrow (h_2,f_2)) \sqcap((h_1,f_1) \Rightarrow (h_3,f_3))$: once again we work only on the second coordinate, which reduces this to the same distributivity for $\Rightarrow_F$ and $\Rightarrow_H$, along with $\tau(x \sqcap_H y) = \tau(x) \sqcap_F \tau(y)$.
\end{enumerate}
\end{proof}
\end{proposition}

\begin{example}
Many more mundane constructions arise as special cases of Artin gluing. For example, given two definable Heyting algebras $H$ and $F$, and taking the definable map $\tau(x) = \top_F$, the Artin gluing $H \ltimes_\tau F$ coincides with the product $H\times F$.
\end{example}

\begin{body}\label{body:gluing-in-one-dim}
Consider a definable gluing map $\tau: H \rightarrow F$ between a one-dimensional definable Heyting algebra $H$ and a finite Heyting algebra $F$. Then $H \ltimes_\tau F$ has infinitely many elements, and as a definable subset $H \times F$, satisfies $1 \leq \dim (H \ltimes_\tau F) \leq \dim H + \dim F = 1 + 0 = 1$. Proposition~\ref{prop:artin-algebra} provides a rich source of definable Heyting algebras in dimension one. For example, taking $([0,1], \leq)$ as a one-dimensional bounded total order, every finite Heyting algebra $F$ and definable monotone map $\tau : [0,1] \rightarrow F$ with $\tau(1) = \top_F$ immediately gives a one-dimensional Heyting algebra $[0,1] \ltimes_\tau F$. Artin gluing also shows that, unlike in a finite Heyting algebra, in a one-dimensional Heyting algebra one cannot always write every element as a finite join of join-irreducibles: consider $[0,1] \ltimes_\tau \{0,1\}$ with the map $\tau$ that sends $0$ to $0$ and every other element to $1$.
\end{body}

\subsection*{Cohesive towers}

\begin{body}
One can turn every finite distributive lattice $(F,\sqcap_F,\sqcup_F,\bot_F,\top_F)$ into a finite Heyting algebra in exactly one way, by setting $x \Rightarrow_F y = \bigsqcup_F \SetComp{z \in F}{x \sqcap_F z \sqsubseteq_F y}$. In general, the same construction fails for infinite bounded distributive lattices (even \textit{complete} distributive lattices where the supremum on the right always exists). As Example~\ref{example:non-heyting-complete-lattice} shows, such examples exist among definable distributive lattices already in dimension one.
\end{body}

\begin{example}\label{example:non-heyting-complete-lattice}
Consider the definable set $H = (\left[0,1\right) \times \{0,1\}) \cup \{(1,1)\}$. Equip $H$ with the pointwise operations
\begin{align*}
    & \sqcap : H^2 \rightarrow H \\
    & (x_1,y_1) \sqcap (x_2,y_2) = (\min\{x_1,x_2\},\min\{y_1,y_2\}) \\
    & ~ \\
    & \sqcup : H^2 \rightarrow H \\
    & (x_1,y_1) \sqcup (x_2,y_2) = (\max\{x_1,x_2\},\max\{y_1,y_2\}).
\end{align*}
Setting $\top = (1,1)$, $\bot=(0,0)$ we obtain a bounded distributive lattice $(H,\sqcap,\sqcup,\bot,\top)$. For each pair of elements $x,y \in H$, the set $\SetComp{z \in H}{x \sqcap z \sqsubseteq y}$ has a supremum. However, defining $x \Rightarrow y$ as that supremum, the structure $(H,\sqcap,\sqcup,\Rightarrow,\bot,\top)$ does not satisfy the Heyting algebra axioms.
\begin{proof}
Compute $(0,1) \Rightarrow (0,0)$. Since $(0,1) \sqcap (a,b) \sqsubseteq (0,0)$ precisely if $\min\{0,a\} = 0$ (which always holds) and $\min\{1,b\} = b = 0$, the set $\SetComp{(a,b) \in H}{(0,1) \sqcap (a,b) \sqsubseteq (0,0)}$ contains elements of the form $(x,0)$ for arbitrarily large $x \in [0,1)$. Thus, it has to have supremum $(0,1) \Rightarrow (0,0) = (1,1)$. But then $$ (0,1) \sqcap ((0,1) \Rightarrow (0,0)) = (0,1) \sqcap (1,1) = (0,1) $$
while $(0,1) \sqcap (0,0) = (0,0)$, so the law $x \sqcap (x \Rightarrow y) = x \sqcap y$ cannot hold.
\end{proof}
\end{example}

\begin{body}
Artin gluings of the chain $([0,1], \leq)$ onto finite algebras give rise to a rich family of one-dimensional infinite Heyting algebras. The resulting algebras, when regarded as distributive lattices, always admit a projection to $([0,1],\leq)$. In fact, the attentive reader will have noticed that all examples of one-dimensional distributive lattices covered so far admit such a projection. This motivates the construction of \textit{cohesive towers} (Lemma~\ref{lemma:cohesive-tower}), bounded distributive lattices that have simple combinatorial descriptions (the \textit{cohesive blueprints} of Definition~\ref{def:cohesive-blueprint}) in terms of projection to a one-dimensional linear order.
\end{body}

\begin{definition}\label{def:cohesive-blueprint}
A \textit{cohesive blueprint} consists of
\begin{itemize}
    \item positive integers $n$ and $N$,
    \item a definable subset $M \subseteq R$ that has a minimum and maximum element,
    \item a decomposition of $M$ into points and intervals $M_1 < \dots < M_n$,
    \item for each $i \in \{1,\dots,n\}$, a finite distributive lattice structure $(H_i, \sqsubseteq_i)$ where $H_i \subseteq \{1,\dots,N\}$,
    \item for each $i \leq j \in \{1,\dots,n\}$, a morphism of bounded-below distributive lattices $l_{i,j} : H_i \rightarrow H_j$ (called the \textit{gluing morphisms}),
\end{itemize}
subject to the following \textit{cohesion conditions}:
\begin{enumerate}
    \item for each $i\leq j\leq k$, the gluing morphisms satisfy $l_{i,k} = l_{j,k} \circ l_{i,j}$, so that in particular $l_{i,i} \circ l_{i,i} = l_{i,i}$; and
    \item for each $i \leq j$, the map $r_{j,i} : H_j \rightarrow H_i$ defined by $r_{j,i}(x)= \bigsqcup_i \SetComp{y \in H_i}{l_{i,j}(y) \sqsubseteq_j x}$, satisfies the cohesion identities $$r_{j,i}(l_{i,j}(p)) = p \sqcup_i r_{j,i}(\bot_j) \text{ and } l_{i,j}(r_{j,i}(q)) = q \sqcap_j l_{i,j}(\top_i)$$ for all $p \in H_i$, $q \in H_j$.
\end{enumerate}
\end{definition}

\begin{lemma}\label{lemma:cohesive-tower}
Take a cohesive blueprint as in Definition~\ref{def:cohesive-blueprint}. Identify $M \subseteq R$ with $M'=\SetComp{(i,m) \in \{1,\dots,n\} \times M}{m \in M_i}$, and using this identification define the set $$H=\SetComp{(i,m,p) \in M' \times \{1,\dots,N\}}{p \in H_i}.$$
Assuming (w.l.o.g.) that $m_1 \leq m \in R$, define the commutative operations
$$
(i,m_1,p) \sqcup (j,m,q) = \begin{cases}
    (j,m,p \sqcup_i q) & \text{if }m_1 = m \\
    (j,m,l_{i,j}(p) \sqcup_j q) & \text{otherwise}
\end{cases}
$$
and
$$
(i,m_1,p) \sqcap (j,m,q) = \begin{cases}
    (i,m_1,p \sqcap_i q) & \text{if }m_1 = m \\
    (i,m_1,p \sqcap_i r_{j,i}(q)) & \text{otherwise}
\end{cases}
$$
Setting $\bot = (1,\min M, \bot_1)$ and $\top = (n,\max M, \top_n)$, call the structure $(H,\sqcap,\sqcup,\bot,\top)$ the \textit{cohesive tower} determined by the cohesive blueprint. Then $(H, \sqcap, \sqcup, \bot, \top)$ is a definable bounded distributive lattice.
\begin{proof}
The description contains finitely many fiber structures and hence finitely many possible gluing morphisms between them. Hence all of these are definable, and the definability of the structure on $H$ follows.

We first prove that for each $i \leq j$, the map $r_{j,i}: H_j \rightarrow H_i$ is a morphism of bounded-above distributive lattices. By Proposition~\ref{prop:adjoint-functor-thm}, $r_{j,i}$ is right adjoint to $l_{i,j}$, and hence preserves $\sqcap_j$ and $\top_j$. We need to demonstrate that $r_{j,i}$ preserves joins. Note that we have
\begin{align*}
    l_{i,j}(r_{j,i}(x \sqcup_j y))
    &= (x \sqcup_j y) \sqcap_j l_{i,j}(\top_i) & (\text{cohesion}) \\
    &= (x \sqcap_j l_{i,j}(\top_i)) \sqcup_j (y \sqcap_j l_{i,j}(\top_i)) & (\text{distributivity}) \\
    &= l_{i,j}(r_{j,i}(x)) \sqcup_j l_{i,j}(r_{j,i}(y)) &  (\text{cohesion}) \\
    &= l_{i,j}(r_{j,i}(x) \sqcup_i r_{j,i}(y)) &  (l_{i,j} \text{ homomorphism}) \\
\end{align*}
for any $x,y \in H_j$. Thus, $l_{i,j}(r_{j,i}(x \sqcup_j y)) \sqsubseteq_j l_{i,j}(r_{j,i}(x) \sqcup_i r_{j,i}(y))$. Using the adjunction between $l_{i,j}$ and $r_{j,i}$ we get that
\begin{align*}
    r_{j,i}(x \sqcup_j y) 
    &\sqsubseteq_i r_{j,i}(l_{i,j}(r_{j,i}(x) \sqcup_i r_{j,i}(y))) &  \\
    &= r_{j,i}(x) \sqcup_i r_{j,i}(y) \sqcup_i r_{j,i}(\bot_j) & (\text{cohesion}) \\
    &= r_{j,i}(x) \sqcup_i r_{j,i}(y). & (r_{j,i}(\bot)\sqsubseteq_i r_{j,i}(y))
\end{align*}
Similarly we have $l_{i,j}(r_{j,i}(x) \sqcup_i r_{j,i}(y)) \sqsubseteq_j l_{i,j}(r_{j,i}(x \sqcup_j y))$
and hence
\begin{align*}
    r_{j,i}(x) \sqcup_i r_{j,i}(y)
    &\sqsubseteq_i r_{j,i}(l_{i,j}(r_{j,i}(x \sqcup_j y))) &  \\
    &= r_{j,i}(x \sqcup_j y) \sqcup_i r_{j,i}(\bot_j) & (\text{cohesion}) \\
    &= r_{j,i}(x \sqcup_j y) & (r_{j,i}(\bot)\sqsubseteq_i r_{j,i}(x \sqcup_j y))
\end{align*}
meaning that $r_{j,i}(x \sqcup_j y) = r_{j,i}(x) \sqcup_i r_{j,i}(y)$ as claimed.

One can check by routine calculations that $(H,\sqcup,\bot,\top)$ constitutes a bounded join-semilattice. As an example we write out the most difficult case, associativity for three elements $(1,m_1,p_1) \sqsubseteq (2,m_2,p_2) \sqsubseteq (3,m_3,p_3)$. We choose these specific integers to denote the fibers purely to simplify our notation, we do not use any property apart from $1 < 2 < 3$. On one side 
\begin{align*}
    ((1,m_1,p_1) \sqcup (2,m_2,p_2)) \sqcup (3,m_3,p_3)
    &= (2,m_2,l_{1,2}(p_1)\sqcup_2 p_2) \sqcup (3,m_3,p_3) &  \\
    &= (3,m_3,l_{2,3}(l_{1,2}(p_1)\sqcup_2 p_2) \sqcup_3 p_3) &  \\
    &= (3,m_3,l_{2,3}(l_{1,2}(p_1))\sqcup_3 l_{2,3}(p_2) \sqcup_3 p_3) &  \\
    &= (3,m_3,l_{1,3}(p_1)\sqcup_3 l_{2,3}(p_2) \sqcup_3 p_3)
\end{align*}
using the gluing identity $l_{1,3} = l_{2,3} \circ l_{1,2}$ and $l_{2,3}$ being a homomorphism. On the other side,
\begin{align*}
    (1,m_1,p_1) \sqcup ((2,m_2,p_2) \sqcup (3,m_3,p_3))
    &= (1,m_1,p_1) \sqcup (3,m_3,l_{2,3}(p_2) \sqcup_3 p_3) &  \\
    &= (3,m_3,l_{1,3}(p_1)\sqcup_3 l_{2,3}(p_2) \sqcup_3 p_3).
\end{align*}
Notice that $r_{j,i}\circ r_{k,j}$ and $ r_{k,i}$ are both right adjoints of $l_{i,k}$. Since right adjoints are unique, we have $r_{j,i}\circ r_{k,j} = r_{k,i}$ and the identities for $(H, \sqcap, \bot, \top)$ follow by dual arguments.

This leaves checking distributivity. Again, we give the representative difficult cases: the others require only minor variations. Consider four elements $$(1,m_1,p_1) \sqsubseteq (2,m_2,p_2) \sqsubseteq (3,m_3,p_3) \sqsubseteq (4,m_4,p_4).$$ Note that we have $(2,m_2,p_2) \sqcup (3,m_3,p_3) = (3,m_3,l_{2,3}(p_2) \sqcup p_3)$.

\vspace{0.5em} \textbf{Case 1.} $(1,m_1,p_1) \sqcap ((2,m_2,p_2) \sqcup (3,m_3,p_3))$. We need only compute the last coordinate. In this case we need to prove
$$ p_1 \sqcap_1 r_{3,1}(l_{2,3}(p_2) \sqcup_3 p_3) = (p_1 \sqcap_1 r_{2,1}(p_2)) \sqcup_1 (p_1 \sqcap_1 r_{3,1}(p_3)) $$
and applying $r_{2,1} \circ r_{3,2} = r_{3,1}$, combined with the join-preservation of $r_{2,1}$ and the cohesion identities, we get
\begin{align*}
  p_1 \sqcap_1 r_{3,1}(l_{2,3}(p_2) \sqcup_3 p_3)
  &= p_1 \sqcap_1 ( r_{3,1}(l_{2,3}(p_2)) \sqcup_1 r_{3,1}(p_3) )  \\
  &= p_1 \sqcap_1 ( r_{2,1}(r_{3,2}(l_{2,3}(p_2))) \sqcup_1 r_{3,1}(p_3) )  \\
  &= p_1 \sqcap_1 ( r_{2,1}(p_2\sqcup_2r_{3,2}(\bot_3)) \sqcup_1 r_{3,1}(p_3) )  \\
  &= p_1 \sqcap_1 ( r_{2,1}(p_2) \sqcup_1 r_{2,1}(r_{3,2}(\bot_3)) \sqcup_1 r_{3,1}(p_3) )  \\
  &= p_1 \sqcap_1 ( r_{2,1}(p_2) \sqcup_1 r_{3,1}(\bot_3) \sqcup_1 r_{3,1}(p_3) )  \\
  &= p_1 \sqcap_1 ( r_{2,1}(p_2) \sqcup_1 r_{3,1}(p_3) )  \\
  &= (p_1 \sqcap_1 r_{2,1}(p_2)) \sqcup_1 (p_1 \sqcap_1 r_{3,1}(p_3))
\end{align*}
as required.

\vspace{0.5em} \textbf{Case 2.} $(2,m_2,p_2) \sqcap ((1,m_1,p_1) \sqcup (3,m_3,p_3))$. Again, we compute the last coordinate. In this case we have the goal
$$p_2 \sqcap_2 r_{3,2}(l_{1,3}(p_1) \sqcup_3 p_3) = l_{1,2}(r_{2,1}(p_2) \sqcap_1 p_1) \sqcup_2 (p_2 \sqcap_2 r_{3,2}(p_3)) $$
and the calculation
\begin{align*}
  p_2 \sqcap_2 r_{3,2}(l_{1,3}(p_1) \sqcup_3 p_3)
  &= p_2 \sqcap_2 ( r_{3,2}(l_{1,3}(p_1)) \sqcup_2 r_{3,2}(p_3) )  \\
  &= p_2 \sqcap_2 ( r_{3,2}(l_{2,3}(l_{1,2}(p_1))) \sqcup_2 r_{3,2}(p_3) )  \\
  &= p_2 \sqcap_2 ( l_{1,2}(p_1) \sqcup_2 r_{3,2}(\bot_3) \sqcup_2 r_{3,2}(p_3) )  \\
  &= p_2 \sqcap_2 ( l_{1,2}(p_1) \sqcup_2 r_{3,2}(p_3) )  \\
  &= (p_2 \sqcap_2 l_{1,2}(p_1)) \sqcup_2 (p_2 \sqcap_2 r_{3,2}(p_3))
\end{align*}
combined with
\begin{align*}
  l_{1,2}(r_{2,1}(p_2) \sqcap_1 p_1) \sqcup_2 (p_2 \sqcap r_{3,2}(p_3))
  &= (l_{1,2}(r_{2,1}(p_2)) \sqcap_2 l_{1,2}(p_1)) \sqcup_2 (p_2 \sqcap_2 r_{3,2}(p_3))  \\
  &= ( (p_2 \sqcap_2 l_{1,2}(\top_1) \sqcap_2 l_{1,2}(p_1)) \sqcup_2 (p_2 \sqcap r_{3,2}(p_3))  \\
  &= (p_2 \sqcap_2 l_{1,2}(p_1)) \sqcup_2 (p_2 \sqcap_2 r_{3,2}(p_3))
\end{align*}
yields the equality of the two sides.
\end{proof}
\end{lemma}

\begin{body}
At this point, it is a useful exercise to prove that if $F$ is a finite Heyting algebra, then each Artin gluing of the form $[0,1] \ltimes_\tau F$ is isomorphic to a cohesive tower over a cohesive blueprint. The reader familiar with the Hall-Dilworth construction may also prove that it arises as a special case of cohesive towers by taking $M=\{0,1\}$.
\end{body}

\subsection*{Principal chain}

\begin{body}
Having seen some examples of bounded distributive lattices and Heyting algebras in dimension one, 
we now develop the combinatorial structure theory of one-dimensional distributive lattices and classify them up to isomorphism. In Theorem~\ref{thm:tame-variation}, we will show that \textit{every} one-dimensional distributive lattice is described by a cohesive blueprint. To do this, we set up a definable projection of the lattice onto a homogeneous linear order (the \textit{principal chain}), then prove that the different fibers of this projection can interact with each other only in very constrained ways.
\end{body}

\begin{definition}\label{def:principal-chain}
Consider a definable distributive lattice $(H,\sqsubseteq)$ with $\dim H = 1$. We call the definable set
$$C_H = \SetComp{x \in H}{\forall y\sqsubset x.\dim[y,x]_\sqsubseteq = 1}$$
the \textit{principal chain of the lattice $H$}.
\end{definition}

\begin{lemma}\label{lemma:principal-chain-is-chain}
In a definable distributive lattice $(H,\sqsubseteq)$ with $\dim H = 1$, the principal chain $C_H$ of $H$ is linearly ordered by $\sqsubseteq$.
\begin{proof}
Take $x,y \in C_H$. If $x \sqcap y \neq x$ and $x \sqcap y \neq y$, then by the defining property of $C_H$ we get that $\dim [x \sqcap y,x]_\sqsubseteq=1$ and $\dim [x \sqcap y,y]_\sqsubseteq=1$. Proposition~\ref{prop:diamond-split} gives a definable (hence dimension-preserving) isomorphism
$$[x \sqcap y, x \sqcup y]_\sqsubseteq \cong [x \sqcap y,x]_\sqsubseteq \times[x \sqcap y,y]_\sqsubseteq.$$
By Chapter~4,~Corollary~1.6(iii)~of~\cite{vandendries-tame}, $$\dim ([x \sqcap y,x]_\sqsubseteq \times[x \sqcap y,y]_\sqsubseteq) = \dim [x \sqcap y,x]_\sqsubseteq + \dim [x\sqcap y,y]_\sqsubseteq = 2,$$
which gives $1 = \dim H \geq \dim [x \sqcap y, x \sqcup y]_\sqsubseteq = 2$, a contradiction.
\end{proof}
\end{lemma}

\begin{body}
In Corollary~\ref{cor:principal-right-adjoint} below we show that the inclusion map $\iota: C_H \hookrightarrow H$, which constitutes a homomorphism of bounded-below distributive lattices for any definable distributive lattice $(H,\sqsubseteq)$ with $\dim H = 1$, admits a right adjoint. It follows (Corollary~\ref{cor:principal-chain-homogeneous}) that $(C_H, \sqsubseteq)$ is always a homogeneous one-dimensional linear order. 
\end{body}

\begin{proposition}\label{prop:calh-is-finite}
For any $x\in H$, the definable set $\mathcal{H}_x = \SetComp{h \in H}{h \sqsubseteq x \wedge \dim [h,x]_\sqsubseteq = 0}$ has only finitely many elements in any definable distributive lattice $(H,\sqsubseteq)$ satisfying $\dim H = 1$.
\begin{proof}
For each $h \in \mathcal{H}_x$, the definable set $[h,x]_\sqsubseteq$ is finite. Use the uniform finiteness property (Chapter~3, Lemma 2.13~of~\cite{vandendries-tame}) to find a single $N \in \mathbb{N}$ so that $\#[h,x]_\sqsubseteq \leq N$ for any $h$. Stratify $\mathcal{H}_x$ into $N$ layers, $\mathcal{H}_{x,0}$ to $\mathcal{H}_{x,N-1}$ as follows:
\begin{enumerate}
    \item $\mathcal{H}_{x,0} = \{x\}$,
    \item for $k \leq N-2$, $\mathcal{H}_{x,k+1} = \SetComp{h \in \mathcal{H}_x}{\exists h'\in\mathcal{H}_{x,k} \text{ s.t. } h' \text{ covers } h}$
\end{enumerate}
and since $\#[h,x]_\sqsubseteq \leq N$ every $h \in \mathcal{H}_x$ belongs to $\mathcal{H}_{x,k}$ for some $k < N$. Assume for a contradiction that $\mathcal{H}_x$ has infinitely many elements. Thus, at least one of the stratification sets $\mathcal{H}_{x,1}, \dots, \mathcal{H}_{x,N-1}$ also has infinitely many elements. Call the index of the least such element $K+1$. Then some element of $\mathcal{H}_{x,K}$ necessarily covers infinitely many elements of $\mathcal{H}_{x,K+1}$, which contradicts Proposition~\ref{prop:finite-cover-property}.
\end{proof}
\end{proposition}

\begin{lemma}\label{lemma:principal-projection}
In a definable distributive lattice $(H,\sqsubseteq)$ with $\dim H = 1$, for each $x \in H$ we can find a unique $c_x \in C_H$ so that $\dim [c_x, x]_\sqsubseteq = 0$ and $\forall c \in C_H. c \sqsubseteq x \leftrightarrow c \sqsubseteq c_x$. Moreover, the map $x \mapsto c_x$ is a definable function.
\begin{proof}
Without loss of generality assume that $x \not\in C_H$. In Proposition~\ref{prop:calh-is-finite} we have shown that the set $\mathcal{H}_x$ has finitely many elements for any $x \in H$. Thus, we can find some $\sqsubseteq$-minimal element $m \in \mathcal{H}_x$. Assume for a contradiction that $m \not\in C_H$. This means that $\neg \forall y \sqsubset m. \dim [y,m]_\sqsubseteq = 1$, or in other words $\exists y \sqsubset m. \dim [y,m]_\sqsubseteq = 0$. Taking such a $y \in H$, we get that $y \sqsubset m \sqsubseteq x$, and using the fact that $[y,x]_\sqsubseteq$ embeds definably into $[y,m]_\sqsubseteq \times [m,x]_\sqsubseteq$ we conclude $\dim [y,x]_\sqsubseteq \leq \dim ([y,m]_\sqsubseteq \times [m,x]_\sqsubseteq) = 0$, and so $y \in \mathcal{H}_x$. This contradicts the $\sqsubseteq$-minimality of $m$ in $\mathcal{H}_x$. Thus, $m \in \mathcal{H}_x \cap C_H$.

By Lemma~\ref{lemma:principal-chain-is-chain}, the relation $\sqsubseteq$ linearly orders the set $\mathcal{H}_x \cap C_H$. From Proposition~\ref{prop:calh-is-finite} we know that $\mathcal{H}_x \cap C_H$ is finite, and we have just shown that it is non-empty. Therefore, $\mathcal{H}_x \cap C_H$ has some $\sqsubseteq$-maximum element $c_x$, which clearly satisfies the given requirements. 

The element $c_x$ is the unique maximal element of $\mathcal{H}_x \cap C_H$, and $\mathcal{H}_x$ admits a definition uniform in $x$ by the usual fiber dimension definability argument (Proposition~1.5~in~Chapter~4~of~\cite{vandendries-tame}). Thus, $x \mapsto c_x$ is definable as claimed.
\end{proof}
\end{lemma}

\begin{proposition}\label{prop:principal-finite-preimages}
Take a definable bounded distributive lattice $(H, \sqsubseteq)$ with $\dim H = 1$, and let $c_\bullet : H \rightarrow C_H$ denote the definable projection map $x \mapsto c_x$ constructed in Lemma~\ref{lemma:principal-projection}. The preimage $c^{-1}_\bullet(c)$ has finitely many elements for every $c \in C_H$.
\begin{proof}
If $c_x = c$, then by construction $c \in \mathcal{H}_x \cap C_H$, so in particular $\dim[c,x] = 0$. Thus, by the uniform finiteness property (Chapter~3, Lemma 2.13~of~\cite{vandendries-tame}), we can find a single $N \in \mathbb{N}$ so that for any $x \in c^{-1}_\bullet(c)$ we have $\#[c,x] \leq N$. A stratification argument similar to that in the proof of Proposition~\ref{prop:calh-is-finite} shows that $c^{-1}_\bullet(c)$ is finite.
\end{proof}
\end{proposition}

\begin{lemma}\label{lemma:principal-chain-prime}
Take a definable distributive lattice $(H,\sqsubseteq)$ with $\dim H = 1$ and some element $c \in C_H$ in its principal chain. For any $a,b \in H$, if $c \sqsubseteq a \sqcup b$, then $c \sqsubseteq a$ or $c \sqsubseteq b$.
\begin{proof}
Assume for a contradiction that $c \sqsubseteq a \sqcup b$, and yet both $c \sqsubseteq a$ and $c \sqsubseteq b$ fail, or equivalently that $a \sqcap c \neq c$ and $b \sqcap c \neq c$. Since $c \in C_H$, the $\sqsubseteq$-intervals $[a \sqcap c, c]_\sqsubseteq$ and $[b \sqcap c, c]_\sqsubseteq$ both have dimension one. Moreover, the elements $a \sqcap c$ and $b \sqcap c$ are incomparable, since $(a \sqcap c) \sqcup (b \sqcap c) = c \neq a \sqcap c$. Proposition~\ref{prop:diamond-split} gives
$$[a \sqcap b \sqcap c, c]_\sqsubseteq = [(a \sqcap c) \sqcap (b \sqcap c), c]_\sqsubseteq \cong [a \sqcap b \sqcap c, a \sqcap c]_\sqsubseteq \times [b \sqcap a \sqcap c, b \sqcap c]_\sqsubseteq.$$
Since $[a \sqcap b \sqcap c, a \sqcap c]_\sqsubseteq \cong [b \sqcap c, c]_\sqsubseteq$ and $[b \sqcap a \sqcap c, b \sqcap c]_\sqsubseteq \cong [a \sqcap c, c]_\sqsubseteq$ hold definably, 
and definable isomorphisms preserve dimension, we get that $\dim [a \sqcap b \sqcap c, a \sqcap c]_\sqsubseteq = 1$ and $\dim [b \sqcap a \sqcap c, b \sqcap c]_\sqsubseteq = 1$, and thus, by Chapter~4,~Corollary~1.6(iii)~of~\cite{vandendries-tame},
$$ \dim [a \sqcap b \sqcap c, c]_\sqsubseteq = \dim [a \sqcap b \sqcap c, a \sqcap c]_\sqsubseteq + \dim [b \sqcap a \sqcap c, b \sqcap c]_\sqsubseteq = 2,$$
which contradicts $\dim H = 1$.
\end{proof}
\end{lemma}

\begin{proposition}\label{prop:principal-projection-joins}
In each definable bounded distributive lattice $(H,\sqsubseteq)$ with $\dim H = 1$, the projection map $c_\bullet : H \rightarrow C_H$ is a homomorphism of bounded-below distributive lattices to the structure $(C_H, \sqsubseteq)$.
\begin{proof}
Routine arguments show that the map preserves meets and $\bot$. Here we prove preservation of joins. Take any $x,y \in H$. Observe that $c_x \in C_H$ and $c_x \sqsubseteq x \sqsubseteq x \sqcup y$, so by Lemma~\ref{lemma:principal-projection}, $c_x \sqsubseteq c_{x \sqcup y}$ always holds. Similarly we always have $c_y \sqsubseteq c_{x \sqcup y}$.

It suffices to show that for any $c \in C_H$, the inequality $c \sqsubseteq c_x \sqcup c_y$ holds precisely if $c \sqsubseteq c_{x \sqcup y}$ does.

So first assume that $c \sqsubseteq c_x \sqcup c_y$. From this assumption, Lemma~\ref{lemma:principal-chain-prime} gives that $c \sqsubseteq c_x$ or $c \sqsubseteq c_y$, and by the observation above, we get $c \sqsubseteq c_{x \sqcup y}$ in either case.

Now assume that $c \sqsubseteq c_{x \sqcup y}$. By Lemma~\ref{lemma:principal-projection}, we have $c \sqsubseteq x \sqcup y$, and so by Lemma~\ref{lemma:principal-chain-prime} we have either $c \sqsubseteq x$ or $c \sqsubseteq y$. Going through Lemma~\ref{lemma:principal-projection} again, we get $c \sqsubseteq c_x \sqsubseteq c_x \sqcup c_y$ or $c \sqsubseteq c_y \sqsubseteq c_x \sqcup c_y$ as required. It follows that $c_{x\sqcup y} = c_x \sqcup c_y$.
\end{proof}
\end{proposition}

\begin{corollary}\label{cor:principal-right-adjoint}
In a definable bounded distributive lattice $(H,\sqsubseteq)$ whose underlying set $H$ satisfies $\dim H = 1$, all of the following hold:
\begin{enumerate}
    \item the principal chain $C_H$ is a bounded-below distributive lattice,
    \item the inclusion homomorphism of bounded-below distributive lattices $\iota: C_H \hookrightarrow H$ admits the definable right adjoint $c_\bullet : H \rightarrow C_H$, and
    \item the preimage $c^{-1}_\bullet(x) \subseteq H$ is a finite distributive sublattice for each $x \in C_H$.
\end{enumerate}
\end{corollary}

\begin{corollary}\label{cor:principal-chain-homogeneous}
The principal chain $(C_H,\sqsubseteq)$ of a definable bounded distributive lattice $(H,\sqsubseteq)$ is homogeneous as a linear order.
\begin{proof}
Take any elements $x,y \in C_H$ satisfying $x \sqsubset y$. We prove that $\dim (C_H \cap [x,y]_\sqsubseteq) = 1$. By definition of the principal chain, $\dim [x,y]_\sqsubseteq = 1$. Proposition~\ref{prop:principal-projection-joins} ensures that the function $c_\bullet$ preserves order. Thus, for any $z \in [x,y]_\sqsubseteq$, we have $x = c_x \sqsubseteq c_z \sqsubseteq c_y = y$, and therefore $c_\bullet([x,y]_\sqsubseteq) \subseteq C_H \cap [x,y]_\sqsubseteq$. By Proposition~\ref{prop:principal-finite-preimages}, the fibers of the function $c_\bullet$ are finite, and by Corollary 1.6.(ii) of Chapter 4 in \cite{vandendries-tame}, $\dim (C_H \cap [x,y]_\sqsubseteq) \geq \dim c_\bullet([x,y]_\sqsubseteq) = \dim [x,y]_\sqsubseteq = 1$ as well.
\end{proof}
\end{corollary}

\subsection*{Fibers of the principal chain}

\begin{proposition}\label{prop:fiber-uniform-bound}
Take a definable bounded distributive lattice $(H,\sqsubseteq)$  with $\dim H = 1$. Let  $c_\bullet : H \rightarrow C_H$ denote the right adjoint to the inclusion map arising from $C_H \subseteq H$. There exists a single $N \in \mathbb{N}$ so that the preimage $c_\bullet^{-1}(c)$ has at most $N$ elements for any $c \in C_H$.
\begin{proof}
Immediate from the uniform finiteness property (Chapter~3, Lemma 2.13~of~\cite{vandendries-tame}) and Corollary~\ref{cor:principal-right-adjoint}.
\end{proof}
\end{proposition}

\begin{proposition}\label{prop:isomorphism-decomp}
Consider a definable bounded distributive lattice $(H, \sqsubseteq)$. When $\dim H = 1$, the principal chain $C_H$ decomposes into finitely many singleton points and intervals such that on each point and interval, the isomorphism type of the preimage $c_\bullet^{-1}$ is constant.
\begin{proof}
Let $N \in \mathbb{N}$ denote the uniform bound of Proposition~\ref{prop:fiber-uniform-bound}. Since there are finitely many distributive lattices with no more than $N$ elements, this means that $c_\bullet^{-1}$ takes on finitely many isomorphism types. Enumerating these isomorphism types as $\{1,\dots,n\}$, we conclude the definability of the function $f : C_H \rightarrow \{1,\dots,n\}$ which sends each $x \in C_H$ to the index of the isomorphism type of $c_\bullet^{-1}(x)$. By Corollaries~\ref{cor:homogeneous-decomposition}~and~\ref{cor:principal-chain-homogeneous}, $C_H$ decomposes into finitely many singleton points and intervals, on each of which the function $f$ is constant.
\end{proof}
\end{proposition}

\begin{example}\label{example:nonuniqueness-of-isomorphism-decomp}
Equip the set $S=[0,1] \times \{0,1\}$ with the product order relation $\sqsubseteq_P$ given by $$(h_1,f_1) \sqsubseteq_P (h_2,f_2) \leftrightarrow (h_1 \leq h_2 \wedge f_1 \leq f_2).$$ Let $\leq_2$ denote the usual lexicographic order on $R^2$ restricted to $S$. Finally, consider the order relation $\sqsubseteq$ on $S$ defined so that $(h_1,f_1)\sqsubseteq (h_2,f_2)$ precisely if one of the following holds:
\begin{itemize}
    \item $h_1 < 1$, $h_2 < 1$ and $(h_1,f_1) \sqsubseteq_P (h_2,f_2)$; or else
    \item $h_1 = 1$ or $h_2 = 1$, and $(h_1,f_1) \leq_2 (h_2,f_2)$.
\end{itemize}
Then $H=(S,\sqsubseteq)$ and $P=(S,\sqsubseteq_P)$ are both definable bounded distributive lattices of dimension one. Moreover, in the decomposition of Proposition~\ref{prop:isomorphism-decomp}, the isomorphism type of $c_\bullet^{-1}$ is identical and constant (the two-element Boolean algebra) on both $C_H$ and $C_P$. Nonetheless, we do not have $H \cong P$, since $(0,1) \sqcup_P (1,0) = (1,1) = \top_P$ holds in $P$, while only $\bot$ and $\top$ have complements in $H$.
\end{example}

\begin{body}
As Example~\ref{example:nonuniqueness-of-isomorphism-decomp} shows, the principal chain $C_H$ and the preimages of the $c_\bullet$ map of Corollary~\ref{cor:principal-right-adjoint} do not determine the structure of a definable distributive lattice up to isomorphism. Creating a decomposition which allows us to recover the structure requires some technicalities. Instead of the isomorphism types of definable bounded distributive lattices from Proposition~\ref{prop:isomorphism-decomp}, we have to work with the more rigid \textit{lexicographic fiber types}. This in turn allows us to define \textit{lexicographic gluing types}, the key to a structure-recovering decomposition. For the remainder of this section, if $x$ belongs to the principal chain $C_H$ of the definable bounded distributive lattice $(H,\sqsubseteq)$, we call the set $c_\bullet^{-1}(x)$ the \textit{principal fiber above $x$}, and denote the principal fiber above $x \in C_H$ as $H_x$.
\end{body}

\begin{definition}\label{def:lex-order}
In what follows, $\leq_k$ denotes the \textit{lexicographic order} on the set $R^k$.
\end{definition}

\begin{proposition}\label{prop:lex-coordinates}
Consider a definable bounded distributive lattice $(H,\sqsubseteq)$  with $H \subseteq R^k$ and $\dim H = 1$, and let $N$ denote the uniform bound on the cardinality of the principal fibers. Let $f_n : C_H \rightarrow H$ denote the partial function that assigns to each $x \in C_H$ the $n$th smallest element of the principal fiber $H_x$ in the lexicographic order $\leq_k$. Then each of the \textit{lexicographic coordinate functions} $f_1,\dots,f_{N}$ are definable.
\end{proposition}

\begin{definition}\label{def:lex-types}
In a definable bounded distributive lattice $(H,\sqsubseteq)$ with $H \subseteq R^k$ and $\dim H = 1$, let $f_1,\dots,f_N$ denote the lexicographic coordinate functions of $H$. Let $H_x$ denote the principal fiber above the chain element $x \in C_H$. The lexicographic coordinate functions identify each principal fiber with a subset of $\{1, \dots, N\}$, so a finite table of coordinate indices allows one to keep track of the join operation in the fiber. More formally, to each $x \in C_H$ we assign an $N\times N$ integer table (equivalently, a vector in $R^{N \cdot N}$) whose $(i,j)$th entry is
\begin{enumerate}
    \item $k$ if $f_i(x), f_j(x)$ and $f_k(x)$ are all defined, and $f_i(x) \sqcup f_j(x) = f_k(x)$,
    \item $0$ otherwise (when one of the coordinates is not defined).
\end{enumerate}
We call this table the \textit{lexicographic fiber type over $x$}.

Given another element $y \in C_H$ so that $x \sqsubset y$, we assign to the pair $(x,y)$ an $N$-integer sequence whose $i$th entry is
\begin{enumerate}
    \item $j$ if $f_i(x)$ is defined and $y \sqcup f_i(x) = f_j(y)$,
    \item $0$ otherwise (when $f_i$ is not defined at $x$).
\end{enumerate}
We call this sequence the \textit{lexicographic gluing type over the pair $(x,y)$}. 
\end{definition}

\begin{proposition}\label{prop:lex-determination}
Taken together, the following information completely determines the structure of a definable bounded distributive lattice $(H,\sqsubseteq)$ with $\dim H = 1$:
\begin{enumerate}
    \item the principal chain $(C_H, \sqsubseteq)$ as a linear order,
    \item the principal chain map $c_\bullet : H \rightarrow C_H$,
    \item the lexicographic fiber type over $H_x$ for each $x \in C_H$, and
    \item the lexicographic gluing type over $(x,y)$ for each $x \sqsubset y \in C_H$.
\end{enumerate}
\begin{proof}
It suffices to show that this information allows one to recover $x \sqcup y$ for any $x, y \in H$ since one can define the other bounded lattice operations from $\sqcup$ alone.

First we deal with the case where $c_x = c_y$. In this case $c_{x\sqcup y} = c_x \sqcup c_y = c_y$, so $x,y$ and $x\sqcup y$ all belong to the fiber $H_y$. The lexicographic fiber type of $H_y$ allows one to define the join of $x$ and $y$ in $H_y$ using the lexicographic coordinates. By Corollary~\ref{cor:principal-right-adjoint} the fibers constitute sublattices of $H$, and so the join of $x$ and $y$ in $H_y$ coincides with $x \sqcup y$.

Now we deal with the case where (w.l.o.g.) $c_x \sqsubset c_y$. Then $c_{x \sqcup y} = c_x \sqcup c_y = c_y$ as well, so $x \sqcup y \in H_y$. Since $c_y \sqsubseteq y$, we have $c_y \sqcup y = y$. This means that $x \sqcup y = x \sqcup (c_y \sqcup y) = (x \sqcup c_y) \sqcup y$. The lexicographic gluing type of $(c_x,c_y)$ allows one to identify $x \sqcup c_y \in H_y$, and as above one can define $(x \sqcup c_y) \sqcup y$ directly using the lexicographic fiber type of $H_y$.
\end{proof}
\end{proposition}

\begin{body}
In Lemma~\ref{lemma:lex-composition}, we prove that lexicographic gluing types \textit{compose}: given $x \sqsubset y \sqsubset z \in C_H$ in a bounded distributive lattice $(H,\sqsubseteq)$, the gluing types over $(x,y)$ and $(y,z)$ together uniquely, definably determine the gluing type over $(x,z)$.
\end{body}

\begin{definition}\label{def:lex-composition}
Consider a definable bounded distributive lattice $(H, \sqsubseteq)$ with $\dim H = 1$ and three elements $x \sqsubset y \sqsubset z$ of its principal chain $C_H$. Let $\ell_{x,y}(n)$ denote the $n$th entry of the lexicographic gluing type over $(x,y)$, and similarly $\ell_{y,z}(n)$ denote the $n$th entry of the lexicographic gluing type over $(y,z)$. We define the \textit{composite lexicographic gluing type over $(y,z)$ and $(x,y)$} as the $N$-integer sequence whose $n$th entry is
\begin{enumerate}
    \item $\ell_{y,z}(\ell_{x,y}(n))$ if $\ell_{x,y}(n) > 0$ and $\ell_{y,z}(\ell_{x,y}(n)) > 0$,
    \item 0 otherwise (when one of the entries above equals zero)
\end{enumerate}
and denote this composite as $\ell_{y,z} \circ \ell_{x,y}$.
\end{definition}

\begin{lemma}\label{lemma:lex-composition}
Given a definable bounded distributive lattice $(H, \sqsubseteq)$ with $\dim H = 1$ and any $x \sqsubset y \sqsubset z \in C_H$, the composite lexicographic gluing type over $(y,z)$ and $(x,y)$ coincides with the lexicographic gluing type over the pair $(x,z)$.
\begin{proof}
Let $\ell_{i,j}$ denote the lexicographic gluing type over the pair $(i,j)$. We must show that the $n$th entry of $\ell_{y,z} \circ \ell_{x,y}$ coincides with that of $\ell_{x,z}$. We straightforwardly verify that the sequences agree on their $n$th entries for each $n$.
Set $a = \ell_{x,y}(n)$ and assuming $a > 0$, define $b = \ell_{y,z}(a)$. First consider the case when $a = 0$. In this case $f_n(x)$ is not defined, and consequently the $n$th entry of the lexicographic gluing type over $(x,z)$ is $0$ as required.
The case $a > 0, b = 0$ cannot occur, since, expanding the definition of lexicographic gluing types of $(x,y)$ and $(y,z)$, this would mean that
\begin{enumerate}
    \item $f_n(x)$ is defined,
    \item $c_y \sqcup f_n(x) = f_a(y)$, but
    \item $f_a(y)$ is not defined.
\end{enumerate}
Finally consider the case when $a > 0$ and $b > 0$ as well. Then, expanding the definition of lexicographic gluing types,
\begin{enumerate}
    \item $f_n(x)$ is defined,
    \item $c_y \sqcup f_n(x) = f_a(y)$,
    \item $f_a(y)$ is defined, and
    \item $c_z \sqcup f_a(y) = f_b(z)$.
\end{enumerate}
It follows that
\begin{align*}
    c_z \sqcup f_n(x) &= (c_z \sqcup c_y) \sqcup f_n(x) &(\text{since } y \sqsubseteq z) \\
    &= c_z \sqcup (c_y \sqcup f_n(x)) & \\
    &= c_z \sqcup f_a(y) & \text{(by 2 above)}\\
    &= f_b(z) & \text{(by 4 above)}
\end{align*}
and combining $c_z \sqcup f_n(x) = f_b(z)$ with 1 and 3, the definition of the lexicographic gluing type of $(x,z)$ forces the $n$th entry to coincide with $b$ as required.
\end{proof}
\end{lemma}

\begin{proposition}\label{prop:lex-definability}
Take a definable bounded distributive lattice $(H,\sqsubseteq)$ with $H \subseteq R^k$ and $\dim H = 1$, with $N$ denoting the uniform bound of Proposition~\ref{prop:fiber-uniform-bound}. Fix some enumeration $\{1,\dots, (N+1)^{N\cdot N}\} \subseteq R$ of the lexicographic fiber types that may occur in the fibers above $H$. The function that assigns to each $x \in C_H$ the index of its lexicographic fiber type in this enumeration is definable. Under an analogous enumeration, the function that associates to each pair $x \sqsubset y \in C_H$ the index of the lexicographic gluing type over $(x,y)$ is definable.
\end{proposition}

\begin{body}
We use Lemma~\ref{lemma:lex-composition} to decompose bounded distributive lattices with dimension one into finitely many well-behaved intervals and points. We perform this decomposition below in several, increasingly technical phases, with Theorem~\ref{thm:tame-variation} giving the final result.
\end{body}

\begin{proposition}\label{prop:lex-decomposition-basic}
Given a definable bounded distributive lattice $(H,\sqsubseteq)$ with $\dim H = 1$, one can decompose the principal chain $C_H$ into finitely many $\sqsubseteq$-intervals $J_1,\dots,J_n$ and points, so that on each interval $J_i$ the lexicographic fiber type function is constant.
\begin{proof}
Follows from Corollary~\ref{cor:homogeneous-decomposition} applied to the lexicographic fiber type function, whose definability we established in Proposition~\ref{prop:lex-definability}. 
\end{proof}
\end{proposition}

\begin{proposition}\label{prop:lex-decomposition-advanced}
Take a definable bounded distributive lattice $(H,\sqsubseteq)$ with $\dim H = 1$ and an $\sqsubseteq$-interval $K \subseteq C_H$ on which the lexicographic fiber type function takes a constant value. We can decompose $K$ into finitely many $\sqsubseteq$-intervals and points $K_1,\dots,K_n$, so that for each of the $\sqsubseteq$-intervals $K_i$ there is a lexicographic gluing type $e$ (depending on $i$) satisfying the following conditions:
\begin{enumerate}
    \item \textit{Composition of gluing types:} the lexicographic gluing type $e$ satisfies the equality $e \circ e = e$ under composition of gluing types,
    \item \textit{Locally unique gluing types:} for each $c \in K_i$ we can find an interval $(c,r)_\sqsubseteq$, so that the lexicographic gluing type of $(c,d)$ is $e$ for each $d \in (c,r)_\sqsubseteq$.
\end{enumerate}
\begin{proof}
Fix some $c \in K$. Let $\ell_c : \SetComp{y \in K}{c \sqsubset y} \rightarrow \{1, \dots, (N+1)^N\}$ denote the map with $\ell_c(y) = \ell_{c,y}$. The definability of this map follows from Proposition~\ref{prop:lex-definability}. By Corollary~\ref{cor:homogeneous-decomposition}, we can decompose $\SetComp{y \in K}{c \sqsubset y}$ into $\sqsubseteq$-intervals and points on which $\ell_c$ is constant. Some $\sqsubseteq$-interval in this decomposition has $c$ as its left endpoint. Let $r(c)$ denote the corresponding right endpoint. We have deduced the existence of a unique lexicographic gluing type $e(c)$ so that $\forall d \in K. (c \sqsubset d) \wedge (d \sqsubset r(c)) \rightarrow \ell_{c,d} = e(c)$. Moreover, the map $c \mapsto e(c)$ is definable with finite codomain. Applying Corollary~\ref{cor:homogeneous-decomposition} once again, we obtain finitely many intervals, points and empty sets $K_1, K_2, \dots $ so that each interval $K_i$ satisfies a local gluing condition of the form
$$ \forall c \in K_i. \forall d \in (c, r(c))_\sqsubseteq. \: \ell_{c,d} = e_i.$$
Now we show that on each interval $K_i$, the corresponding lexicographic gluing type $e_i$ satisfies $e_i \circ e_i = e_i$ under composition of gluing types. Consider an element $x \in K_i$, and a corresponding $y \in (x,r(x))_\sqsubseteq$, and a $z \in (y,r(y) \sqcap r(x))_\sqsubseteq$. We then have both $\ell_{x,y} = e_i$ and $\ell_{y,z} = e_i$, as well as $\ell_{x,z} = e_i$ using the fact that $z \sqsubset r(x)$. By Lemma~\ref{lemma:lex-composition}, we conclude $e_i \circ e_i = \ell_{y,z} \circ \ell_{x,y} = \ell_{x,z} = e_i$ as claimed.
\end{proof}
\end{proposition}

\begin{proposition}\label{prop:lex-decomposition-expert}
Take a definable bounded distributive lattice $(H,\sqsubseteq)$ with $\dim H = 1$, a $\sqsubseteq$-interval $L \subseteq C_H$, and a lexicographic gluing type $e$ of $H$. Assume the following conditions hold:
\begin{enumerate}
    \item \textit{Constant fiber type}: the lexicographic fiber type function takes a constant value on $L$,
    \item \textit{Composition of gluing type:} the equality $e \circ e = e$ under composition of gluing types,
    \item \textit{Locally unique gluing type}: for each $c \in L$ we can find an interval $(c,r)_\sqsubseteq$, so that the lexicographic gluing type of $(c,d)$ is $e$ for each $d \in (c,r)_\sqsubseteq$.
\end{enumerate}
Then one can decompose $L$ into a union of finitely many $\sqsubseteq$-intervals and points $L_1, L_2, \dots$ so that for each $\sqsubseteq$-interval $L_i$ and $x \sqsubset y \in L_i$, the lexicographic gluing type of $(x,y)$ coincides with $e$.
\begin{proof}
Say that an element $d \in L$ \textit{subsumes} $c \in L$ when $c \sqsubset d$ and $\ell_{c,y} = e$ holds for any $y$ with $c \sqsubset y \sqsubseteq d$. Define
$$ s(c) = \sup \SetComp{d \in L}{d \text{ subsumes } c},$$
if necessary taking this supremum in the definable order completion of $L$. By Proposition~\ref{prop:homogeneous-order-embedding} the principal chain embeds into $(R, \leq)$ itself, so one can even regard $s$ as a definable function with codomain $R\cup\{\pm\infty\}$. Note that the set of subsuming elements is downward-closed. If $c \sqsubset y \sqsubset s(c)$, then we can find some $y \sqsubset y' \sqsubset s(c)$ that subsumes $c$, so in fact $y$ subsumes $c$. 

\vspace{0.5em}\textbf{Claim 1.} We claim that whenever $d$ subsumes $c$, we in fact have $s(c) = s(d)$. Clearly we have $s(d) \sqsubseteq s(c)$. We prove $s(c) \sqsubseteq s(d)$ by showing that any $x \in L$ with $x \sqsubset s(c)$ satisfies $x \sqsubseteq s(d)$. This holds by transitivity when $x \sqsubseteq d$, so from here onward assume $d \sqsubset x$. Take any $y \in L$ between $d$ and $x$, then choose $z \in L$ small enough so that $z \sqsubset y$ and the locally unique gluing type condition applies, giving $\ell_{d,z} = e$. Thus we have
$$ d \sqsubset z \sqsubset y \sqsubseteq x$$
and we can argue that
\begin{align*}
    \ell_{d,y} &= \ell_{z,y} \circ \ell_{d,z} & (\text{Lemma } \ref{lemma:lex-composition}) \\
    &= \ell_{z,y} \circ e & (\text{by } \ell_{d,z} = e) \\
    &= \ell_{z,y} \circ (e \circ e) & (\text{assumption 2})\\
    &= (\ell_{z,y} \circ e) \circ e & \\
    &= (\ell_{z,y} \circ \ell_{d,z}) \circ e & (\text{by } \ell_{d,z} = e) \\
    &= \ell_{d,y} \circ e & (\text{Lemma } \ref{lemma:lex-composition}) \\
    &= \ell_{d,y} \circ \ell_{c,d} & (d \text{ subsumes } c)\\
    &= \ell_{c,y} & (\text{Lemma } \ref{lemma:lex-composition}) \\
    &= e & (y \text{ subsumes } c)
\end{align*}
holds for such an $x \in L$. Thus any $x \in L$ with $d \sqsubset x \sqsubseteq s(c)$ subsumes $d$, and therefore $x \sqsubseteq s(d)$ as required.

\vspace{0.5em}\textbf{Claim 2.} We claim that the definable function $s$ takes only finitely many values. To see this, consider the definable family $s^{-1}(c) = \SetComp{x \in L}{s(x) = c}$ as $c$ varies over $s(L)$. We have $\dim s^{-1}(s(L)) = \dim L = 1$ since $L$ is a $\sqsubseteq$-interval in the homogeneous one-dimensional total order $(C_H, \sqsubseteq)$. Moreover, by the locally unique gluing type condition, each fiber $s^{-1}(c)$ has dimension 1. But then, by Corollary 1.8. of Chapter 4 in \cite{vandendries-tame}, $\dim s(L) + 1 = \dim s^{-1}(s(L)) = \dim L = 1$. Consequently $\dim s(L) = 0$, which shows the finiteness of $s(L)$. 

\vspace{0.5em}
\textbf{Claim 3.} Using Claim 2 and Corollary~\ref{cor:homogeneous-decomposition}, decompose $L$ into finitely many $\sqsubseteq$-intervals and points $L_1, \dots, L_n$ so that the function $s$ takes a constant value on each of these $\sqsubseteq$-intervals. We claim that on each $\sqsubseteq$-interval $L_i$, for any $x \sqsubset y \in L_i$, the lexicographic gluing type $\ell_{x,y}$ coincides with $e$. We have $x \sqsubseteq y \sqsubset s(x)$. Since $(y,s(x))_\sqsubseteq$ is an interval in the homogeneous linear order $C_H$, we get that $\dim (y,s(x))_\sqsubseteq = 1$, and can choose some $z$ so that $y \sqsubset z \sqsubset s(x)$. Then some $z'$ with $z \sqsubseteq z' \sqsubset s(x)$ subsumes $x$ by definition of $s(x)$, and $y \in (x,z')_\sqsubseteq$, so $\ell_{x,y} = e$ holds as claimed.
\end{proof}
\end{proposition}

\begin{theorem}\label{thm:tame-variation}
Given a definable bounded distributive lattice $(H,\sqsubseteq)$ with $\dim H = 1$, one can decompose the principal chain $C_H$ into the following data:
\begin{itemize}
    \item finitely many $\sqsubseteq$-cells ($\sqsubseteq$-intervals and points) $M_1,\dots,M_n$,
    \item a definable function $t$ that assigns to each natural number index $1 \leq i \leq n$ a lexicographic fiber type $t_i$ of $H$,
    \item a definable function $l$ that assigns to each pair of natural number indices $1 \leq i \leq j \leq n$ a lexicographic gluing type $l_{i,j}$
\end{itemize}
 so that the following hold:
\begin{enumerate}
    \item the $\sqsubseteq$-cells are linearly ordered, so that when $1 \leq i < j \leq n$ holds, $\forall x \in M_i. \forall y \in M_j. x \sqsubset y$,
    \item the lexicographic gluing types $l_{i,i}$ satisfy $l_{i,i} \circ l_{i,i} = l_{i,i}$ for $1 \leq i \leq n$,
    \item the lexicographic fiber type function is constant on each cell, meaning that if $x,y \in M_i$, then both $x$ and $y$ have lexicographic fiber type $t_i$, and
    \item the cells of the decomposition determine the lexicographic gluing types, in the sense that if $x \in M_i$ and $y \in M_j$ for $x \sqsubset y$ and $i \leq j$, then $\ell_{x,y} = l_{i,j}$.
\end{enumerate}
\begin{proof}
Let $M_1, \dots, M_n$ denote the decomposition obtained by applying Proposition~\ref{prop:lex-decomposition-basic} to $C_H$, then Propositions~\ref{prop:lex-decomposition-advanced}~and~\ref{prop:lex-decomposition-expert} to each interval of the resulting decompositions. We know from Proposition~\ref{prop:lex-decomposition-basic} that the lexicographic fiber type function is constant on each such cell, and from Proposition~\ref{prop:lex-decomposition-expert} that \textit{within} each of the resulting $\sqsubseteq$-intervals and points, only one lexicographic gluing type occurs. Denoting this gluing type $l_{i,i}$, Proposition~\ref{prop:lex-decomposition-expert} also guarantees that $l_{i,i} \circ l_{i,i} = l_{i,i}$. We now show that even \textit{between} points that lie in two different $\sqsubseteq$-cells, cell membership completely determines the lexicographic gluing type. Take cell indices $i < j$, $x \in M_i$ and $y \in M_j$. We prove only that if $M_i$ is an $\sqsubseteq$-interval, then for $x' \in M_i$ with $x \sqsubset x'$, we have $\ell_{x',y} = \ell_{x,y}$. The various other cases follow by analogous arguments. Using the fact that the $\sqsubseteq$-interval $M_i$ has no largest element, fix some $x'' \in M_i$ with $x \sqsubset x' \sqsubset x''$, then argue as follows:
\begin{align*}
    \ell_{x',y} &= \ell_{x'',y} \circ \ell_{x',x''} & (\text{Lemma } \ref{lemma:lex-composition}) \\
    &= \ell_{x'',y} \circ l_{i,i} & (\text{by } x', x'' \in M_i)\\
    &= \ell_{x'',y} \circ l_{i,i} \circ l_{i,i} & (\text{by } l_{i,i} \circ l_{i,i} = l_{i,i})\\
    &= \ell_{x'',y} \circ \ell_{x',x''} \circ l_{i,i} & (\text{by } x', x'' \in M_i)\\
    &= \ell_{x',y} \circ l_{i,i} & (\text{Lemma } \ref{lemma:lex-composition})\\
    &= \ell_{x',y} \circ \ell_{x,x'} & (\text{by } x, x' \in M_i)\\
    &= \ell_{x,y}. & (\text{Lemma } \ref{lemma:lex-composition})
\end{align*}
\end{proof}
\end{theorem}

\begin{corollary}\label{cor:tame-variation-blueprints}
Every definable bounded distributive lattice $(H, \sqsubseteq)$ with $\dim H = 1$ is isomorphic to a cohesive tower over a cohesive blueprint.
\end{corollary}

\begin{body}
Theorem~\ref{thm:tame-variation} gives an essentially complete classification of one-dimensional distributive lattices definable in o-minimal structures expanding a real-closed field. Proposition~\ref{prop:homogeneous-order-embedding} lets us regard the principal chain as a subset of $R$, and specifying a cell decomposition of $R$, along with a set of lexicographic fiber and gluing data, lets us reconstruct every one-dimensional distributive lattice as a cohesive tower over a \textit{canonical} cohesive blueprint.
\end{body}

\section{Birkhoff-style results}\label{sec:birkhoff-style-results}

\begin{body}
By classic results of Birkhoff and Stone, one can embed every bounded distributive lattice into a Boolean algebra. This result cannot admit a fully general definable version in the o-minimal setting, due to the shortage of infinite definable Boolean algebras (Proposition~\ref{prop:boolean-algebras-are-finite}). As we have seen in Example~\ref{example:non-heyting-complete-lattice}, unlike finite distributive lattices, distributive lattices of dimension one may fail to have a Heyting algebra structure. In this section, we prove that at least every definable bounded distributive lattice of dimension one \textit{embeds} definably into a definable Heyting algebra (Theorem~\ref{thm:birkhoff-embedding}).
\end{body}

\subsection*{Covariant regular representation}

\begin{body}
Any poset $(P, \sqsubseteq)$ embeds into the Heyting algebra $(\downarrow P, \subseteq)$ of downward-closed subsets of $P$, with the lattice meets (resp. joins) corresponding to set intersections (resp. set unions). This \textit{covariant regular representation} constitutes a key ingredient in Birkhoff's representation theorem.
\end{body}

\begin{theorem}[Birkhoff]\label{thm:birkhoff}
Consider a finite distributive lattice (equivalently, finite Heyting algebra) $(H,\sqsubseteq)$. Let $J(H)$ denote the set of \textit{join-irreducible elements} of $(H, \sqsubseteq)$ with the usual convention that $\bot \not\in J(H)$. The covariant regular representation $(\downarrow J(H), \subseteq)$ of the poset $(J(H), \sqsubseteq)$ is isomorphic to $(H, \sqsubseteq)$ as a Heyting algebra.
\end{theorem}

\begin{body}
For infinite distributive lattices $L$, the conclusion of Theorem~\ref{thm:birkhoff} can fail maximally. In Boolean algebras, atoms and join-irreducibles coincide, so for an infinite, atomless Boolean algebra $B$, one obtains the triviality of the lattice $\downarrow J(B)$.
\end{body}

\begin{body}
Below, we give a definable analogue of the covariant regular representation for those definable posets which have bounded width in the strong sense of Definition~\ref{def:chain-decomposition}. 
\end{body}

\begin{definition}\label{def:dcr-rep}
Consider a definable poset $(P, \sqsubseteq_P)$. A \textit{definable covariant regular representation} $(L,E)$ of $(P, \sqsubseteq_P)$ consists of the following data:
\begin{itemize}
    \item a definable set $L$ (the \textit{code set}), and
    \item a definable binary relation $E \subseteq P \times L$ (\textit{membership}),
\end{itemize}
subject to the following conditions:
\begin{enumerate}
    \item for each code $x \in L$, the fiber $E_x = \SetComp{y \in P}{yEx}$ yields a downward-closed subset of $P$,
    \item for every definable downward-closed $S \subseteq P$ we can find a unique code $x \in L$ so that $E_x = S$.
\end{enumerate}
\end{definition}

\begin{proposition}\label{prop:dcr-embedding}
Take a definable poset $(P, \sqsubseteq_P)$ with a definable covariant regular representation $(L,E)$. Then we can find a definable relation $\sqsubseteq$ on $L$ making $(L,\sqsubseteq)$ into a Heyting algebra so that the poset $(P, \sqsubseteq_P)$ embeds definably into $(L, \sqsubseteq)$ via the map $\downarrow\! x$ that takes each $x \in P$ to the $(L,E)$-code of the set $\SetComp{y \in P}{y \sqsubseteq x}$.
\begin{proof}
For codes $c,d \in L$, define $c \sqsubseteq d$ to abbreviate $\forall x \in P. xEc \rightarrow xEd$. The existence and uniqueness clauses of Definition~\ref{def:dcr-rep} yield that $(L, \sqsubseteq)$ is a bounded distributive lattice isomorphic to the lattice of definable downward-closed subsets of $P$. We can embed $(P, \sqsubseteq_P)$ into $(L, \sqsubseteq)$ by mapping each $p \in P$ to the code of the set $\downarrow\! p = \SetComp{x \in P}{x \sqsubseteq p}$. To show the definability of this map, write it as the definable set
$$ \SetComp{(p,l) \in P \times L}{\forall x.xEl\leftrightarrow x \sqsubseteq p}.$$
We claim that the lattice of definable downward-closed subsets of $P$ forms a Heyting algebra with the implication operation $B \Rightarrow C$ given by the set $\SetComp{p \in P}{\downarrow\!p \cap B \subseteq C}$. It suffices to show that for downward-closed subsets $A,B,C$ of $P$, we have $A \subseteq B \Rightarrow C$ precisely if $A \cap B \subseteq C$.

First, assume that $A \subseteq B \Rightarrow C$, and take any $x \in A \cap B$. We want to show $x \in C$. Since $x \in A$, by assumption $x \in B \Rightarrow C$ as well, in other words $\downarrow\!x \cap B \subseteq C$. But $x \sqsubseteq x$ and $x \in B$, so $x \in \downarrow\!x \cap B$, and therefore $x \in C$ as required.

For the other direction assume that $A \cap B \subseteq C$, and take any $x \in A$. We want to show that $x \in B \Rightarrow C$, or in other words that $\downarrow\!x \cap B \subseteq C$. Take any $y \in \downarrow\!x \cap B$. We then need to show that $y \in C$ as well. But $y \sqsubseteq x$ and $x \in A$: by downward-closure of $A$, we get that $y \in A$. Therefore, $y \in A \cap B$ holds, and by assumption so does $y \in C$ as required. This shows that the lattice of definable downward-closed subsets of $P$ forms a Heyting algebra.
\end{proof}
\end{proposition}

\begin{body}
Note that the embedding into the definable covariant regular representation given in Proposition~\ref{prop:dcr-embedding} preserves all meets that exist in the poset, but need not preserve joins.
\end{body}

\begin{example}
The bounded distributive lattice $([0,1]^2, \sqsubseteq)$, where $(x_1,y_1) \sqsubseteq (x_2,y_2)$ precisely if $x_1 \leq x_2$ and $y_1 \leq y_2$, does not admit definable covariant regular representations.
\end{example} 

\begin{lemma}\label{lemma:dcr-for-linear-orders}
Any definable linear order $(P, \sqsubseteq_P)$ admits a definable covariant regular representation.
\begin{proof}
Identify $(P, \sqsubseteq_P)$ with its image in its definable completion $(\overline{P}, \sqsubseteq)$. Note that $\overline{P}$ has a minimum element $\sup \emptyset$. Let $P'$ denote the set consisting of those elements of $P$ which have an immediate predecessor in $P$. Define the code set $$ L = \SetComp{(x,y) \in \overline{P} \times \{0,1\}}{ (x \in P' \wedge y = 1) \vee (x \in P\setminus P') \vee (x \not\in P \wedge y = 0) }$$
and the membership relation
$$ E = \SetComp{(p,(x,y))\in P \times L}{p \sqsubset x \vee (y = 1 \wedge p = x)}.$$
For a downward-closed subset $S$ of $P$, write the corresponding code $s$ as
$$s = \begin{cases}
    (\sup S, 1) & \text{if } \sup S \in S \\
    (\sup S, 0) & \text{otherwise}
\end{cases}$$
We then have $E_s = S$ as desired. Uniqueness of $s$ follows immediately since each element of $L$ encodes a different subset of $P$.
\end{proof}
\end{lemma}

\begin{definition}\label{def:chain-decomposition}
Consider a definable poset $(P, \sqsubseteq)$ with $P \subseteq R^n$. A \textit{chain decomposition} of $(P, \sqsubseteq)$ consists of finitely many definable cells $T_1, \dots, T_m$, each $T_i \subseteq P$ subject to the following conditions:
\begin{enumerate}
    \item $P = \bigcup_{i} T_i$,
    \item each restriction $(T_i, \sqsubseteq)$ constitutes a linear order.
\end{enumerate}
We call the chain decomposition \textit{disjoint} if $T_i \cap T_j = \emptyset$ for all $i \neq j$.
\end{definition}

\begin{proposition}\label{prop:disjoint-chain-decomposition}
If a definable poset $(P, \sqsubseteq)$ has a chain decomposition, then it also has a disjoint chain decomposition.
\begin{proof}
Given a chain decomposition $T_1, \dots, T_n$ of $(P, \sqsubseteq)$, set 
$S_1 = T_1$ and $S_{i+1} = T_{i+1} \setminus \bigcup_{j\leq i} S_j$ for each $i < n$. Then a cell decomposition partitioning each of $S_1, \dots, S_n$ gives the desired disjoint chain decomposition.
\end{proof}
\end{proposition}

\begin{theorem}\label{thm:dcr-for-chain-decomposed-orders}
Every definable poset $(P, \sqsubseteq)$ with a chain decomposition admits a definable covariant regular representation.
\begin{proof}
Using Proposition~\ref{prop:disjoint-chain-decomposition}, work with a disjoint chain decomposition $T_1, \dots, T_n$ of  $(P, \sqsubseteq)$. Invoke Lemma~\ref{lemma:dcr-for-linear-orders} to obtain definable covariant regular representations $(L(i), E(i))$ for each linearly ordered set $(T_i,\sqsubseteq)$. Define the set of pre-codes
$$L' = \prod_{i} L(i)$$
and the pre-membership relation $xE'(l_1,\dots,l_n)$ for $x \in P$ and $(l_1,\dots,l_n) \in L'$ by the formula
$$ \bigvee_i x \in T_i \wedge(x,l_i) \in E(i).$$
Call a pre-code $l \in L'$ downward-closed if $xE'l$ implies $y E'l $ for all $y \sqsubseteq x$. Downward-closed pre-codes form a definable set, so let the set $L$ of codes consist of all downward-closed pre-codes. Define the membership relation $E$ as the restriction of $E'$ to $L$ in the second argument. It is then clear that each $E_l$ represents a distinct downward-closed subset of $P$. We now construct a representing code $s$ with $E_s = S$ for each definable downward-closed subset $S \subseteq P$. Given such an $S$, the intersection $S \cap T_i$ is downward-closed in each $T_i$. Hence, we can find some code $s(i)$ so that $E(i)_{s(i)} = S \cap T_i$. Let $s = (s(1),\dots,s(n))$. Clearly $s \in L'$. We claim that $E'_s = S$.

First, assume that $x \in E'_s$ for some $x \in P$. Then by definition of pre-membership we can find a unique $i$ so that $x \in T_i$ and $(x,s(i)) \in E(i)$. Since $E(i)_{s(i)} = S \cap T_i$, the latter clause implies $x \in S$.

For the other direction, assume that $x \in S$. Since $T_1, \dots, T_n$ constitutes a disjoint chain decomposition of $P$, we can find a unique $i$ so that $x \in T_i$. But then $x \in S \cap T_i$, and thus $(x,s(i)) \in E(i)$. By definition of pre-membership, $xE's$.

Since $E'_s = S$ and $S$ is a downward-closed subset of $P$, the pre-code $s$ is downward-closed as well, giving $s \in L$ and $E_s = E'_s = S$ as required. 
\end{proof}
\end{theorem}

\subsection*{Join-irreducibly generated completion}

\begin{definition}\label{def:join-irreducibly-generated}
We call a distributive lattice $(H, \sqsubseteq)$ \textit{join-irreducibly generated} (or \textit{jig}) if for any $x \in H$ we can find finitely many join-irreducible elements $i_1,\dots,i_n \in H$ so that $x = i_1 \sqcup i_2 \sqcup \dots \sqcup i_n$. We allow $n=0$ for the empty join which takes the value $\bot_H$.
\end{definition}

\begin{body}
All finite distributive lattices satisfy Definition~\ref{def:join-irreducibly-generated}. However, as we have seen in Section~\ref{body:gluing-in-one-dim}, infinite distributive lattices may fail to be jig, already in dimension one. In this short but technical subsection, we show how to embed such one-dimensional bounded distributive lattices into jig ones in a definable manner.
\end{body}

\begin{theorem}\label{thm:jig-embedding}
Any definable bounded distributive lattice structure $(H,\sqsubseteq)$ of dimension one embeds definably into a jig bounded distributive lattice of dimension one.
\end{theorem}

\begin{body}
One could go through Theorem~\ref{thm:tame-variation}, and prove directly that the cohesive towers of Lemma~\ref{lemma:cohesive-tower} embed into jig bounded distributive lattices. However, when the author tried this approach, it required some tradeoffs in other qualities of the paper (for example, one had to treat definability separately). Ultimately, one finds it easier to do the whole proof in one go.
\begin{itemize}
    \item Throughout this subsection, work in a fixed one-dimensional bounded distributive lattice $(H, \sqsubseteq_H)$.
    \item Let $N$ denote the bound of Proposition~\ref{prop:fiber-uniform-bound} for $(H, \sqsubseteq_H)$, and $f_1,\dots, f_N : C_H \rightarrow H$ the corresponding lexicographic coordinate functions.
    \item Apply Theorem~\ref{thm:tame-variation} to obtain a decomposition of the principal chain $C_H$ into finitely many $\sqsubseteq$-cells $M_1, \dots, M_n$, along with the lexicographic fiber and gluing type functions $t$ and $l$.
\end{itemize}
We will use the data above to construct a one-dimensional jig bounded distributive lattice $(G, \sqsubseteq)$ and a definable distributive lattice embedding $e: H \hookrightarrow G$.
\end{body}

\begin{definition}\label{def:jig-position-chain}
For each $i \in \{1,\dots,n\}$, define the set $\overline{M}_i$ as
$$ \overline{M}_i = \begin{cases}
  \{i\} \times M_i & \text{if } \# M_i = 1\\
  \{(-i,0,\dots,0)\} \cup (\{i\} \times M_i) & \text{if }M_i\text{ is an $\sqsubseteq_H$-interval}
\end{cases}$$
then set $\overline{M}=\bigcup_{i \in \{1,\dots,n\}} \overline{M}_i$. We call $\overline{M}$ the \textit{position chain of $G$}, its elements \textit{chain positions}, and the nonzero integers between $-n$ and $n$ the \textit{chain indices}.
\end{definition}

\begin{definition}\label{def:jig-chain-orders}
For chain indices $x,y$, define the linear order relation $x <_{\overline{M}} y$ as
$$ |x| < |y| \vee (|x| = |y| \wedge x < 0 \wedge 0 < y),$$
so that $-1 <_{\overline{M}} 1 <_{\overline{M}} -2 <_{\overline{M}} 2$ and so on.
Between chain positions $(i_1,m_1)$, define the partial order relation $(i_1,m_1) \sqsubseteq_{\overline{M}} (i_2,m_2)$ to hold precisely if one of the following three cases occurs:
\begin{itemize}
    \item $i_1 <_{\overline{M}} i_2$,
    \item $i_1 = i_2 < 0$, or
    \item $0<i_1=i_2$ and $m_1\sqsubseteq_H m_2$.
\end{itemize}
\end{definition}

\begin{body}
Informally, the construction of Definition~\ref{def:jig-chain-orders} amounts to gluing a new point to the bottom of each $\sqsubseteq_H$-interval $M_i$, placing it just before $M_i$ and just after all the preceding cells. We will realize our completion $(G, \sqsubseteq)$ as a cohesive blueprint over $\overline{M}$.
\end{body}

\begin{definition}\label{def:jig-underlying-set}
Using the chain $\overline{M}$, let the completion have underlying set
$$ G = \SetComp{(i,m,p) \in \overline{M} \times \{1,\dots,N\}}{\exists q \leq N. t_{|i|}(p,q)\neq 0}.$$
Assuming (w.l.o.g.) that $|i| \leq |j|$ and $(i,m_1) \sqsubseteq_{\overline{M}} (j,m_2)$, define the join operation $\sqcup$ on $G$ as
$$(i,m_1,p) \sqcup (j,m_2,q) = \begin{cases}
    \left(j,m_2,t_{|j|}(p,q)\right) & \text{if } (i,m_1)=(j,m_2) \\
    \left(j,m_2,t_{|j|}(l_{|i|,|j|}(p),q)\right) & \text{otherwise}.
\end{cases}
$$
Note that $\sqcup$ is the join operation of a cohesive tower construction over $\overline{M}$,  with the cohesion identities inherited from the original fiber and
gluing data in Theorem~\ref{thm:tame-variation}, and so $(G, \sqsubseteq)$ is a one-dimensional bounded distributive lattice.
\end{definition}

\begin{body}
One could summarize the idea behind Definition~\ref{def:jig-underlying-set} as follows. We adjoin a new fiber below each interval $M_i$. This new fiber not only has the same lexicographic fiber type as $M_i$ (so that operations within the new fiber work the same way as they do in fibers of $M_i$), but has the same lexicographic gluing types to every other fibers as well: the join formula reflects how the fiber and gluing types determine the lattice operations in Proposition~\ref{prop:lex-determination}. Note that in general, the new negative positions in $G$ cause no new algebraic phenomena to emerge: the calculation of joins and meets reduces to the calculation of lexicographic coordinates, and every coordinate calculation that happens in them already happens in positive chain positions inside the same interval cell. Thus, e.g. distributivity in $(G,\sqsubseteq)$ reduces to distributivity in $(H,\sqsubseteq_H)$.
\end{body}

\begin{proposition}\label{prop:jig-embedding-map}
The function $e: H \rightarrow G$ which sends $x \in H$ to the unique $(i,m,p)$ so that $f_p(m) = x$ and $m \in M_i$, is a definable embedding of the bounded distributive lattice $(H,\sqsubseteq_H)$ into $(G,\sqsubseteq)$.
\begin{proof}
We need to show only the preservation of joins. Take two elements $f_p(m) \in H$ and $f_q(n) \in H$, where $m \sqsubseteq_H n$, $m \in M_i$ and $n \in M_j$. As usual, we treat only the harder case $m \neq n$. We then have
\begin{align*}
    e(f_p(m) \sqcup_H f_q(n))
    &= e(f_{t_j(l_{i,j}(p),q)}(n))  \\
    &= (j, n, t_j(l_{i,j}(p),q))  \\
    &= (i,m,p) \sqcup (j, n, q) \\
    &= e(f_p(m)) \sqcup e(f_q(n)) \\
\end{align*}
as claimed.
\end{proof}
\end{proposition}

\begin{proposition}\label{prop:jig-jig}
The bounded distributive lattice $(G,\sqsubseteq)$ is jig.
\begin{proof}
Argue by induction on the chain index of elements. Assume that $(j,m,p) \in G$, and that for any $i <_{\overline{M}} j$, one can write every element with chain index $i$ as a join of join-irreducibles in $(G, \sqsubseteq)$. We show that we can then write $(j,m,p)$ itself as a join of join-irreducibles. Let $b$ denote the $\sqsubseteq$-least element with chain position $(j,m)$, and let $x \sqsubseteq' y$ denote $(j,m,x) \sqsubseteq (j,m,y)$. Then $b$ is $\sqsubseteq'$-minimal. Without loss of generality, assume $(j,m,p)$ does not arise as a nontrivial join of elements with the same chain position. We then have the following cases to consider.

\vspace{0.5em} \textbf{Case 1.} $p = b$. Then $(j,m,b)$ cannot arise as a nontrivial join in $G$. Write $(j,m,b) = (i,m_1,p_1) \sqcup (j',m_2,p_2)$, w.l.o.g. assuming that $|i| \leq |j'|$ and $m_1 \sqsubseteq_{\overline{M}} m_2$. By definition of $\sqcup$, we immediately have that $j' = j$ and $m_2 = m$. We have ruled out $(i,m_1) = (j,m)$ by assumption, so by the second case of the definition of $\sqcup$ we have that $p_2 \sqsubseteq' t_{|j|}(l_{|i|,|j|}(p_1),p_2) \sqsubseteq' b$, and from the $\sqsubseteq'$-minimality of $b$ we conclude $p_2 = b$, showing the join-irreducibility of $(j,m,b) \in G$.

\vspace{0.5em} \textbf{Case 2.} We can find an earlier chain index $i <_{\overline{M}} j$ so that $(j,m,p) = (i,m_1,q) \sqcup (j,m,b)$. Then we can write $(i,m_1,q)$ as a join of join-irreducibles by the inductive hypothesis, and $(j,m,b)$ is join-irreducible by case 1, so we can write $(j,m,p)$ itself as a join of join-irreducibles.

\vspace{0.5em} \textbf{Case 3.} We can write $(j,m,p) = (j,m_1,q) \sqcup (j,m,b)$ for some $m_1 \sqsubset_{\overline{M}} m$. Then we must have $j > 0$ and by the definition of $\sqcup$,
$$(j,m,p) = (j,m, t_j(l_{j,j}(q),b)) = (j,m, l_{j,j}(q)),$$
giving $l_{j,j}(q) = p$. By the second clause of Theorem~\ref{thm:tame-variation}, this means $l_{j,j}(p) = l_{j,j}(l_{j,j}(q)) = l_{j,j}(q) = p$, and therefore that
\begin{align*}
    (-j,0,\dots,0,p) \sqcup (j,m,b) &= (j,m,t_{j}(l_{j,j}(p), b)) & (\text{defn. of } \sqcup) \\
    &=(j,m,t_{j}(p, b)) & (\text{by } l_{j,j}(p) = p)\\
    &=(j,m,p).
\end{align*}
Then we can write $(-j,0,\dots,0,p)$ as a join of join-irreducibles by the inductive hypothesis, and $(j,m,b)$ is join-irreducible by case 1, so we can write $(j,m,p)$ itself as a join of join-irreducibles.

\vspace{0.5em} \textbf{Case 4.} Neither of the previous cases hold. Then $(j,m,p)$ is join-irreducible. If we can write $(j,m,p) = x \sqcup y$ for $x\neq (j,m,p)$ and $y \neq (j,m,p)$, then by definition of $\sqcup$ one of $x,y$ (say the latter) must have chain position $(j,m)$. But then we could write $(j,m,p) = x \sqcup y = x \sqcup (y \sqcup (j,m,b)) = (x \sqcup (j,m,b)) \sqcup y$, which contradicts the assumption that $(j,m,p)$ does not arise as a nontrivial join of elements with the same chain position.

\vspace{0.5em} In each case we have managed to write $(j,m,p)$ as a join of join-irreducible elements in $G$. Hence, by induction, $(G,\sqsubseteq)$ is jig. 
\end{proof}
\end{proposition}

This concludes the proof of Theorem~\ref{thm:jig-embedding}.

\subsection*{Birkhoff representation}

\begin{proposition}\label{prop:chain-decomposition-for-dim-one}
Every definable bounded distributive lattice with $\dim H = 1$ has a chain decomposition.
\begin{proof}
Take a definable bounded distributive lattice $(H, \sqsubseteq)$ with $\dim H = 1$. Use Theorem~\ref{thm:tame-variation} to decompose $C_H$ into finitely many linearly ordered $\sqsubseteq$-cells $M_1,\dots,M_n$. Fix some interval cell $M_i$ and let $f_p$ denote some lexicographic coordinate function defined on $M_i$. It suffices to prove that $(f_p(M_i), \sqsubseteq)$ itself constitutes a linear order. Assume that $x \sqsubseteq y$. We need to prove $f_p(x) \sqsubseteq f_p(y)$: since $M_i$ itself is linearly ordered, the result then follows. Working in lattice notation, we wish to prove $f_p(x) \sqcup f_p(y) = f_p(y)$. We can compute
\begin{align*}
    f_p(x) \sqcup f_p(y)
    &= (x \sqcup f_p(x)) \sqcup (y \sqcup f_p(y)) & (\text{by } x,y \in C_H) \\
    &= (x \sqcup y) \sqcup f_p(x) \sqcup f_p(y) & (\text{associativity}) \\
    &= y \sqcup f_p(x) \sqcup f_p(y) & (\text{by } x \sqsubseteq y) \\
    &= f_{l_{i,i}(p)}(y) \sqcup f_p(y) & (\text{defn. of } l_{i,i}) \\
    &= f_{t_i(l_{i,i}(p),p)}(y) & (\text{defn. of } t_{i})
\end{align*}
so $f_p(x) \sqcup f_p(y) = f_p(y)$ reduces to proving the equality $t_i(l_{i,i}(p),p) = p$.
To prove this, choose some $z \in M_i$ so that $x \sqsubset y \sqsubset z$.

\vspace{0.5em}\textbf{Claim 1.} The equality $f_p(x) \sqcup f_p(y) \sqcup f_p(z) = f_p(x) \sqcup f_p(z)$ holds. Repeating the previous calculation we have that
$$ f_p(x) \sqcup f_p(z) = f_{t_i(l_{i,i}(p),p)}(z)$$
and
$$f_p(x) \sqcup f_p(y) \sqcup f_p(z) = f_{t_i(l_{i,i}(p),p)}(y) \sqcup f_p(z) = f_{t_i(l_{i,i}(t_i(l_{i,i}(p),p),p))}(z),$$ where we can prove the equality of the indices as follows:
\begin{align*}
    t_i(l_{i,i}(t_i(l_{i,i}(p),p),p))
    &= t_i(t_i(l_{i,i}(l_{i,i}(p)),l_{i,i}(p)),p) &  \\
    &= t_i(t_i(l_{i,i}(p),l_{i,i}(p)),p) & (\text{by $l_{i,i} \circ l_{i,i} = l_{i,i}$}) \\
    &= t_i(l_{i,i}(p),p).
\end{align*}
Thus, $f_p(x) \sqcup f_p(y) \sqcup f_p(z) = f_p(x) \sqcup f_p(z) $ holds as claimed.

\vspace{0.5em}\textbf{Claim 2.} Recall that the cells of the decomposition of Theorem~\ref{thm:tame-variation} determine the lexicographic fiber and gluing types. Since we work over a single such cell $M_i$, we can (appealing to Lemma~\ref{lemma:cohesive-tower}) choose a single index $q$ so that $f_p(x) \sqcap f_p(y) = f_q(x)$ and $f_p(y) \sqcap f_p(z) = f_q(y)$. We claim that for this $q$, the equality $p = t_i(l_{i,i}(q),q)$ holds. We know that
\begin{align*}
    f_p(y)
    &= f_p(y) \sqcap (f_p(x) \sqcup f_p(y) \sqcup f_p(z)) & (\text{absorption}) \\
    &= f_p(y) \sqcap (f_p(x) \sqcup f_p(z)) & (\text{Claim 1}) \\
    &= (f_p(y) \sqcap f_p(x)) \sqcup (f_p(y) \sqcap f_p(z))) & (\text{distributivity}) \\
    &= f_q(x) \sqcup f_q(y) & \\
    &= f_{t_i(l_{i,i}(q),q)}(y).
\end{align*}
Equating the indices, the claim follows.

\vspace{0.5em}With the $q$ of Claim 2, we can now calculate
\begin{align*}
    l_{i,i}(p)
    &= l_{i,i}(t_i(l_{i,i}(q),q)) & (\text{Claim 2}) \\
    &= t_i(l_{i,i}(l_{i,i}(q)),l_{i,i}(q)) &  \\
    &= t_i(l_{i,i}(q),l_{i,i}(q))  & (\text{by $l_{i,i} \circ l_{i,i} = l_{i,i}$}) \\
    &= l_{i,i}(q) & (\text{idempotence})
\end{align*}
which yields $l_{i,i}(p) = l_{i,i}(q)$, and at long last, the required equality
\begin{align*}
    t_i(l_{i,i}(p),p)
    &= t_i(l_{i,i}(q),p) &  \\
    &= t_i(l_{i,i}(q),t_i(l_{i,i}(q),q)) & (\text{Claim 2}) \\
    &= t_i(t_i(l_{i,i}(q),(l_{i,i}(q))),q) & (\text{associativity}) \\
    &= t_i(l_{i,i}(q),q) &   \\
    &= p & (\text{Claim 2})
\end{align*}
follows.
\end{proof}
\end{proposition}

\begin{theorem}\label{thm:birkhoff-embedding}
Every definable bounded distributive lattice $(H, \sqsubseteq_H)$ with $\dim H = 1$ embeds definably into a definable Heyting algebra.
\begin{proof}
By Theorem~\ref{thm:jig-embedding}, it suffices to consider the case of jig $H$. Proposition~\ref{prop:chain-decomposition-for-dim-one} guarantees that $(H, \sqsubseteq_H)$, and consequently also $(J(H), \sqsubseteq_H)$, have chain decompositions. Let $(L, E)$ denote the definable covariant regular representation of $(J(H), \sqsubseteq_H)$ given by Theorem~\ref{thm:dcr-for-chain-decomposed-orders}, equipped with the Heyting algebra structure $(L,\sqsubseteq)$ from Proposition~\ref{prop:dcr-embedding}.

We will prove that $(H, \sqsubseteq_H)$ embeds into $(L,\sqsubseteq)$ as a bounded distributive lattice. Define the map $\Downarrow : H \rightarrow L$ by sending each $h \in H$ to the code of the set $\SetComp{j \in J(H)}{j \sqsubseteq_H h}$. This map is definable by the formula
$$ \Downarrow \ = \SetComp{(h,l) \in H \times L}{\forall j \in J(H). jEl\leftrightarrow j\sqsubseteq_H h}. $$
We claim that $\Downarrow$ constitutes an injective homomorphism of bounded distributive lattices. The preservation of meets, bottom and top follows immediately from the definition. Preservation of joins follows from Lemma~\ref{lemma:irreducible-forks}. For injectivity, assume that $\Downarrow x = \Downarrow y$. We need to prove $x = y$. By our assumption, for each $j \in J(H)$, $j \sqsubseteq_H x$ holds precisely if $j \sqsubseteq_H y$. Use jig to write $x = j_1 \sqcup \dots \sqcup j_n$ for join-irreducible $j_i$. Since each $j_i \sqsubseteq x$ satisfies $j_i \sqsubseteq y$ as well, their join $x$ satisfies $x = j_1 \sqcup \cdots \sqcup j_n \sqsubseteq y$ too. Switching the roles of $x$ and $y$, the same argument gives $y \sqsubseteq x$, from which antisymmetry yields $x = y$.  
\end{proof}
\end{theorem}

\section{Applications}\label{sec:applications}

\subsection*{Free Heyting algebras}

\begin{body}
Heyting algebras originally arose to provide an algebraic semantics for intuitionistic logic, a role analogous to that played by Boolean algebras for classical logic. A natural question is how much of this application survives when we restrict our attention to definable Heyting algebras. While well-understood, the free Heyting algebra on a single generator, $F_1$, already has infinitely many elements. Here, using our previous results, we construct a definable Heyting algebra that contains $F_1$ as a subalgebra.
\end{body} 

\begin{example}\label{example:definable-ha-with-free-subalgebra}
Consider the poset $(S, \sqsubseteq_S)$ with underlying set $S = \{0,1\} \times \left[0,+\infty\right)$ and order relation $\sqsubseteq_S$ given by the clauses
\begin{enumerate}
    \item $(x,a)\sqsubseteq_S(x,b)$ precisely if $a \leq b$,
    \item $(0,a)\sqsubseteq_S(1,b)$ precisely if $a+2 \leq b$, and
    \item $(1,a)\sqsubseteq_S(0,b)$ precisely if $a+1 \leq b$,
\end{enumerate}
as $x$ ranges over $\{0,1\}$ and $a,b$ range over $R$. This poset admits a chain decomposition into the subsets $T_0 = \SetComp{(0,a) \in S}{a \in \left[0,+\infty\right)}$ and $T_1 = \SetComp{(1,a) \in S}{a \in \left[0,+\infty\right)}$, so by Theorem~\ref{thm:dcr-for-chain-decomposed-orders} its covariant regular representation $(H, \sqsubseteq)$ constitutes a Heyting algebra. We claim that $F_1$ occurs as a subalgebra of $(H, \sqsubseteq)$.
\begin{proof}
The integer coordinate elements of $S$ constitute a copy of the Rieger-Nishimura ladder, the dual Kripke frame of $F_1$. Hence, in particular the subalgebra of $H$ generated by the element $\downarrow\!(0,0) \in H$ is isomorphic to $F_1$.
\end{proof}
\end{example}

\begin{body}
The underlying definable set of the Heyting algebra constructed in Example~\ref{example:definable-ha-with-free-subalgebra} has dimension two. It follows from Theorem~\ref{thm:tame-variation} that no one-dimensional definable Heyting algebra contains $F_1$ as a subalgebra.
\end{body}

\begin{proposition}\label{prop:cohesive-locally-finite}
Every definable Heyting algebra $(H, \sqsubseteq)$ with $\dim H = 1$ is locally finite.
\begin{proof}
Write $b \in B_H$ if $b \in C_H$ and $b$ belongs to some singleton cell in the decomposition of Theorem~\ref{thm:tame-variation}. It suffices to prove that for all $x,y \in H$, we have $c_{x \Rightarrow y} \in B_H \cup \{c_x,c_y, c_\top\}$, since then the Heyting subalgebra generated by a finite subset is always contained in a finite union of finite fibers. Assume for a contradiction that $c_{x \Rightarrow y}$ instead belongs to an interval cell $M_i$ in the decomposition, but does not coincide with $c_x \in M_j$ or $c_y \in M_k$.

Since $x \sqcap (x \Rightarrow y) = x \sqcap y$ and the map $c_\bullet$ preserves meets, we have $c_x \sqcap c_{x \Rightarrow y}=c_x \sqcap c_y$. By assumption $c_x \sqcap c_{x \Rightarrow y} \neq c_x$, so we must have $c_x \sqsubset c_{x \Rightarrow y}$.

Without loss of generality assign lexicographic coordinates $f_u(c_{x \Rightarrow y}) = x \Rightarrow y$, $f_v(c_x) = x$ and $f_w(c_y) = y$. Using the fact that $M_i$ is an interval cell, take some $b$ such that $c_{x \Rightarrow y} \sqsubset b$. By Theorem~\ref{thm:tame-variation} the lexicographic fiber type is constant on $M_i$, so $f_u(b)$ is defined.

By the formulae in Lemma~\ref{lemma:cohesive-tower}, we have $x \sqcap f_u(b) = f_v(c_x) \sqcap f_u(b) = f_{t_j(v,r_{i,j}(u))}(c_x)$ on one hand, and $x \sqcap (x \Rightarrow y) = f_v(c_x) \sqcap f_u(c_{x\Rightarrow y}) = f_{t_j(v,r_{i,j}(u))}(c_x)$ on the other. So $x \sqcap f_u(b) = x \sqcap (x \Rightarrow y) = x \sqcap y \sqsubseteq y$. By the definition of the implication operation, we get that $f_u(b) \sqsubseteq x \Rightarrow y$, and applying $c_\bullet$ to both sides also that $b \sqsubseteq c_{x \Rightarrow y}$, contradicting $c_{x \Rightarrow y} \sqsubset b$.
\end{proof}
\end{proposition}

\begin{body}
Since $F_2$ already contains infinite antichains of join-irreducibles, the results above tell us precisely which free Heyting algebras can occur as subalgebras of definable Heyting algebras, and in which dimensions.
\end{body}

\subsection*{Satisfiability dimension}

\begin{body}
Many algebraic laws possess ``all-or-nothing'' qualities. For example, by a well-known result of Gustafson~\cite{gustafson-degree-of-commutativity}, in a non-Abelian finite group $G$, the equality $xy=yx$ cannot hold for more than $\frac{5}{8}$ of all possible pairs of elements $(x,y) \in G$. Other laws in the language of group theory admit similar results, e.g. the solution set of $\SetComp{x \in G}{x^2=1}$ either has cardinality $|G|$ or else at most cardinality $\frac{3}{4}|G|$. On the logic side, Evans~\cite{evans-bck-satisfiability,evans-commutingdegreebckalgebras} investigated such \textit{degree of satisfiability} results in BCK-algebras, while Bumpus and Kocsis~\cite{kocsis-degreesat} gave a classification of one-variable equations in Heyting algebras with respect to degree of satisfiability.
\end{body}

\begin{body}
The existence of Haar measure (and more generally finitely additive measures which ``measure index uniformly'') gives rise to a rich theory of \textit{degree of satisfiability} on infinite groups~\cite{antolin-infinite-dc,tointon-commuting,amir-infinite}. Can one develop similar machinery for determining ``largeness'' that keeps working beyond the group-theoretic setting, and can be used to obtain similar results for equations in other algebraic structures? We propose the o-minimal dimension of the solution set as a candidate machinery for this purpose in the tame setting, and apply the results above to obtain a dimension-theoretic counterpart to the classification of one-variable equations in finite Heyting algebras.
\end{body}

\begin{definition}\label{def:dimension-gap}
Consider a first-order language $\mathcal{L}$ and a $\mathcal{L}$-theory $T$. We say that an $\mathcal{L}$-formula $\varphi(x_1,\dots,x_n)$ in $n$ free variables \textit{has dimensional gap} if in any infinite definable $\mathcal{L}$-structure $(M,\dots)$ so that $(M,\dots) \models T$, one of
\begin{itemize}
    \item $\dim M^n  > \dim \SetComp{(x_1,\dots,x_n) \in M^n}{\varphi(x_1,\dots,x_n)}$, or
    \item $M^n = \SetComp{(x_1,\dots,x_n) \in M^n}{\varphi(x_1,\dots,x_n)}$
\end{itemize}
hold.
\end{definition}

\begin{body}
Propositions~\ref{prop:heyting-dimension-gap-c3c4}~and~\ref{prop:heyting-dimension-gap-c5} supply examples of formulas with and without dimensional gap in the context of Heyting algebras. For now, we illustrate the non-vacuity of the definition using a more elementary example from linear algebra. 
\end{body}

\begin{example}
Work in the language $\mathcal{L} = (+,0, \pi,\lambda\cdot|_{\lambda\in R})$ extending the language of $R$-vector spaces with a unary function symbol $\pi$. Consider the theory $T$ extending the theory of $R$-vector spaces with the following axioms:
\begin{enumerate}
    \item $\forall x.\forall y.\pi(x+y) = \pi(x)+\pi(y)$,
    \item $\forall x.\pi(\lambda\cdot x) = \lambda \cdot \pi(x)$ for each $\lambda \in R$,
    \item $\forall x. \pi (\pi(x)) = \pi(x)$.
\end{enumerate}
Then the formula $\pi(x) = 0$ has dimensional gap.
\begin{proof}
Notice that if $\pi(m) \neq 0$ holds for some $m \in M$, then $\mathrm{im}\:\pi$ is infinite, so $\dim (\mathrm{im}\:\pi) \geq 1$. Moreover, definable sets of the form $\SetComp{x \in M}{\pi(x) = p}$ must have the same dimension $k$ for each $p\in\mathrm{im}\:\pi$. Applying Proposition~4.1.5~of~\cite{vandendries-tame} gives
$$ \dim M = \dim \bigcup_{p \in \mathrm{im}\:\pi} \{p\}\times \SetComp{x \in M}{\pi(x) = p} = \dim (\mathrm{im}\:\pi) + k,$$
which we can rearrange to the claimed inequality $k = \dim M - \dim (\mathrm{im}\:\pi) < \dim M$.
\end{proof}
\end{example}

\begin{body}
Intuitively, one should think of properties with dimensional gap as satisfying an ``all-or-nothing'' property: unless they hold everywhere, the set of counterexamples is large. Proposition~\ref{prop:pillay-large} and Lemma~\ref{lemma:pillay-index} illustrate the relationship between Definition~\ref{def:dimension-gap} and the group-theoretic notion of index.
\end{body}

\begin{definition}
We call a subset $H \subseteq G$ in a group $(G, \cdot)$ \textit{large in the sense of group theory} if we can find finitely many elements $g_1, \dots, g_n\in G$ so that the translates $g_1 H, \dots, g_nH$ together cover $G$, i.e. $G = \bigcup_{i \leq n} g_i H$.
\end{definition}

\begin{body}
Notice that a subgroup has finite index precisely if it is large in the sense of group theory.
\end{body}

\begin{lemma}[Pillay~\cite{pillay-groups}]\label{lemma:pillay-index}
A definable subgroup $H$ of a definable group $G$ satisfies $\dim H = \dim G$ precisely if $H$ has finite index in $G$.
\end{lemma}

\begin{proposition}\label{prop:pillay-large}
Take a theory $T$ extending the theory of groups and a one-variable $T$-formula $\varphi(x)$ with dimensional gap in $T$. Then in any model $G$ of $T + \neg\forall x. \varphi(x)$, the set of counterexamples $\SetComp{x \in G}{\neg\varphi(x)}$ is large in the sense of group theory.
\begin{proof}
Follows from Lemma 2.4~of~\cite{pillay-groups}.
\end{proof}
\end{proposition}

\begin{body}
In group theory, one frequently gets true results by replacing ``\textit{finitely presented group}'' with ``\textit{definable group}'', and ``\textit{on a set of positive finitely additive measure}'' with ``\textit{on a set of full dimension}''. As an example, we give Proposition~\ref{prop:virtually-abelian-xsquared}, a straightforward analogue of Theorem~5.1~in~\cite{amir-infinite}.
\end{body}

\begin{proposition}\label{prop:virtually-abelian-xsquared}
Take a definable group $(G,\cdot,\bullet^{-1},1)$. If $\dim \SetComp{x \in G}{x^2 = 1}$ and $\dim G$ coincide, then $G$ is virtually Abelian.
\begin{proof}
Since $G^0$ is always a minimal definable subgroup of finite index in $G$, we can find a coset of $G^0$ that has a full dimension nonempty intersection with $\SetComp{x \in G}{x^2 = 1}$. Pick some $t$ in the interior of this intersection, along with an open neighborhood $V \ni t$. Using $t^2=1$ and the fact that definable bijections preserve dimension, we get that $tV$ is a neighborhood of the identity. For $x \in tV$, we have $tx \in V$, so $(tx)^2=1$ and $txt=x^{-1}$. Considering the preimage of $tV$ with respect to $\bullet^{-1}$ lets us construct a full dimension $W$ with $WW\subseteq tV$. But if $x,y \in W$, then $xy \in tV$, so $$x^{-1} y^{-1} = (txt)(tyt) = txyt = (xy)^{-1} = y^{-1}x^{-1}$$ and $x,y$ commute. This means that the centralizer of any $x \in W$ is a definable subgroup of full dimension. By Lemma~\ref{lemma:pillay-index}, each such centralizer has finite index, so has $G^0$ as a subgroup. Now take an arbitrary $g\in G^0$. Since the centralizer of any element of $W$ contains $g$, we also have $W \subseteq C_G(g)$, and thus $\dim G = \dim W \leq \dim C_G(g)$. One more application of Lemma~\ref{lemma:pillay-index} gives that $C_G(g)$ has finite index in $G$. Again, this means that $G^0$ is contained in $C_G(g)$. Since we chose $g \in G^0$ arbitrarily, any two elements of $G^0$ must commute, hence $G^0$ is Abelian and $G$ is virtually Abelian as claimed.
\end{proof}
\end{proposition}

\begin{body}
Note that dimensional gap behaves better than the natural measure-theoretic notions of largeness available in o-minimal theories. For example, in the measure of~\cite{berarducci-measure}, the lower-dimensional subsets of $[0,1]^n$ are guaranteed to have measure zero. Over the reals $\mathbb{R}$, the converse also holds, and $n$-dimensional subsets have positive measure. However, over a hyperreal field, one can pick an infinitesimal $\varepsilon > 0$, and then $[0,\varepsilon]^n$ constitutes a measure zero set of dimension $n$. Dimensional gap captures sets so small that they remain measure zero even after scaling.
\end{body}

\begin{body}
In what follows, we determine all the one-variable equations (necessarily without parameters) in the language of Heyting algebras which have dimensional gap. Much like in degree of satisfiability, the law of excluded middle ($x \vee \neg x = \top$) exhibits the ``all-or-nothing'' phenomenon (Proposition~\ref{prop:heyting-dimension-gap-c3c4}) in a very strong sense.
\end{body}

\begin{body}
By well-known results of Nishimura and Rieger~\cite{rieger-nishimura-lattice} concerning the free Heyting algebra on one generator (refer e.g. to Section~2~of~\cite{kocsis-degreesat}), one can prove that every system of equations in one free variable in the language of Heyting algebras defines the same set as a single equation of the form $t(x) = \top$ for some term $t(x)$, and falls into one of the following five classes:
\begin{enumerate}
    \item[C1.] trivial (defines the whole set or $\emptyset$),
    \item[C2.] always defines a singleton,
    \item[C3.] always defines the same set as $x \sqcup \neg x = \top$,
    \item[C4.] always defines the same set as $\neg\neg x = x$, or
    \item[C5.] always defines a superset of $\neg\neg x = \top$.
\end{enumerate}
The equations in C1 and C2 have dimensional gap for vacuous reasons. In Proposition~\ref{prop:heyting-dimension-gap-c3c4} we point out that equations in C3 and C4 have dimensional gap since their solution sets are always finite. Finally, in Proposition~\ref{prop:heyting-dimension-gap-c5}, we give a construction showing that the remaining (infinitely many) equations have no dimensional gap.
\end{body}

\begin{proposition}\label{prop:heyting-dimension-gap-c3c4}
The equations $x \sqcup \neg x = \top$ and $\neg \neg x = x$ have dimensional gap.
\begin{proof}
Consider an infinite definable Heyting algebra $(H, \sqsubseteq)$. Check that the set $D_H = \SetComp{x \in H}{\neg\neg x = x}$ constitutes a definable Boolean algebra with the meet operation $\sqcap$ and the join operation $(x,y)\mapsto \neg\neg(x \sqcup y)$. By Proposition~\ref{prop:boolean-algebras-are-finite}, $D_H$ has dimension zero. Since $\SetComp{x \in H}{x \sqcup \neg x = \top} \subseteq D_H$ in any Heyting algebra $H$, both $x \sqcup \neg x = \top$ and $\neg \neg x = x$ have dimensional gap.
\end{proof}
\end{proposition} 

\begin{proposition}\label{prop:heyting-dimension-gap-c5}
Consider a nontrivial one-variable equation $t(x) = \top$ in the language of Heyting algebras, so that $\SetComp{x \in H}{\neg \neg x = \top} \subseteq \SetComp{x \in H}{t(x) = \top}$ in every Heyting algebra $H$. Such an equation does not have dimensional gap.
\begin{proof}
Take such an equation $t(x) = \top$. We construct an infinite definable Heyting algebra where the solution set of $t(x) = \top$ has finite but non-empty complement. In such an algebra, the solution set has maximal dimension.

Using the completeness of intuitionistic propositional logic with respect to finite Heyting algebra semantics, choose a finite Heyting algebra $H$ in which $t(h) \neq \top_H$ for some $h \in H$. Note that $h \neq \top_H$. Invoke Theorem~\ref{thm:birkhoff} to identify $H$ with the covariant regular representation of some finite poset $P$.

Construct a new poset $(Q, \sqsubseteq)$ by taking the ordered sum of $P$ and the interval $([0,1],\leq)$. The resulting infinite poset clearly has a chain decomposition, so apply Theorem~\ref{thm:dcr-for-chain-decomposed-orders} to obtain its definable covariant regular representation $(L, E)$. Identify the codes of $L$ with the definable downward-closed subsets of $Q$, writing $(L, \subseteq)$ for the Heyting algebra structure on $L$. Since $P$ itself constitutes a downward-closed subset of $Q$, identify $H$ with the finite set $\SetComp{l \in L}{l \subseteq P}$, so that $\top_H = P$.

\vspace{0.5em} \textbf{Claim 1.} For any $l \in L$, we have $l \subseteq P$ or $P \subseteq l$. If we do not have $l \subseteq P$, then we have some $i \in [0,1]$ so that $i \in l$. In the ordered sum, $\forall p \in P. p \sqsubseteq i$. Consequently, $P \subseteq l$ as claimed.

\vspace{0.5em} \textbf{Claim 2.}  For any $l \in L$ such that $P \subseteq l$, we have $\neg l = \emptyset$, and consequently $\neg\neg l = \top_L$. As noted in the proof of Proposition~\ref{prop:dcr-embedding}, $\neg l = \SetComp{x \in Q}{\downarrow_Q\! x \cap l = \emptyset}$. Take an arbitrary $x \in Q$. If $x \in P$, then $\downarrow_Q\!x \cap P = \downarrow_Q\! x\neq \emptyset$. If $x \in [0,1]$ instead, then $P \subseteq \downarrow_Q\! x$, so $\downarrow_Q\!x \cap P = P \neq \emptyset$ again. This shows that $x \not\in \neg l$ for an arbitrary $x \in Q$, in other words $\neg l = \emptyset$ as claimed.

\vspace{0.5em} Since $\SetComp{l \in L}{\neg \neg l = \top_L} \subseteq \SetComp{x \in L}{t(x) = \top_L}$, Claim 2 gives that the solution set of the equation $t(x) = \top$ contains at least $\SetComp{l \in L}{P \subseteq l}$. By Claim 1, the complement of the solution set is a subset of the finite set $\SetComp{l \in L}{l \subseteq P}$. Now, one has to verify that this complement is not empty.

First, one can check that the map $f: L \rightarrow H$ given by $f(x) = x \cap P$ is a  definable Heyting algebra homomorphism from $(L,\subseteq)$ to $(H,\subseteq)$: preservation of meets and joins is immediate, and preservation of implication requires only the definition of the implication operation from Proposition~\ref{prop:dcr-embedding}, and the fact that $P$ is a downward-closed subset of $Q$. Since $h \subset P$, this homomorphism satisfies $f(h) = h \cap P = h \neq P$, in other words $f(h) \neq \top_H$. But homomorphisms commute with terms, so $f(t(h)) = t(f(h)) = t(h) \neq \top_H$. Homomorphisms preserve $\top$, but $f(t(h)) \neq \top_H$, and so $t(h) \neq \top_L$. This means that the complement of the solution set $\SetComp{l \in L}{t(l) = \top_L}$ is finite but not empty, and the solution set has full dimension in $L$.
\end{proof}
\end{proposition}

\subsection*{Open questions}

\begin{body}
Theorem~\ref{thm:dcr-for-chain-decomposed-orders} requires its input poset to decompose definably into finitely many linear orders. This condition leads to tedious verifications (Proposition~\ref{prop:chain-decomposition-for-dim-one}) in applications. Especially in light of Proposition~\ref{prop:join-irreducible-antichains}, one finds it much easier to bound the size of the largest antichain (the so-called \textit{width} of the poset) than to construct any specific chain decomposition. Outside the definable setting, this would suffice: Dilworth's theorem states that a poset has finite width $w$ precisely if it has a decomposition into $w$ disjoint linear orders. The author has tried and failed to prove a definable analogue of this result on multiple occasions, leading to the

\vspace{0.5em} \textbf{Question.} Does every definable poset $(P, \sqsubseteq)$ of finite width admit a definable chain decomposition in an o-minimal structure?
\end{body}


\newpage

\bibliography{biblio}
\bibliographystyle{plainurl}

\end{document}